\documentclass[a4paper,10pt]{amsart}
\usepackage{amssymb,amsmath,amsfonts,amsthm}
\usepackage[dvips]{graphicx}
\usepackage{psfrag}

\usepackage{color}

\usepackage{todo}

\theoremstyle{plain}
\newtheorem{main}{Theorem}

\newtheorem{maincor}[main]{Corollary}
\newtheorem{theorem}{Theorem}[section]
\newtheorem{lemma}[theorem]{Lemma}
\newtheorem{proposition}[theorem]{Proposition}
\newtheorem{corollary}[theorem]{Corollary}
\theoremstyle{remark}
\newtheorem{remark}[theorem]{Remark}
\newtheorem{definition}[theorem]{Definition}
\newtheorem{example}[theorem]{Example}

\newcommand{\quand}{\quad\text{and}\quad}
\newcommand{\Leb}{\operatorname{vol}}

\newcommand{\C}{\operatorname{C}}

\newcommand{\Gibbs}{\operatorname{Gibbs}}

\newcommand{\Jac}{\operatorname{Jac}}

\newcommand{\G}{\operatorname{G}}
\newcommand{\Int}{\operatorname{Int}}

           \def\ea{\end{array}}
          \def\ec{\end{center}}
     \def\ed{\end{description}}
        \def\ee{\end{equation}}
       \def\eea{\end{eqnarray}}
     \def\eeaa{\end{eqnarray*}}
 \def\et{\end{thebibliography}}
\def\bib{\bibitem}

\def\Orb{{\rm Orb}}
\def\Diff{{\rm Diff}}
\def\Cl{{\rm Cl}}
\def\Gibb{{\rm Gibbs}}
\def\inv{{\rm inv}}

\def\supp{\operatorname{supp}}

\def\cG{{\mathcal G}}
\def\cA{{\mathcal A}}
\def\cD{{\mathcal D}}

\def\cO{{\mathcal O}}
\def\cI{{\mathcal I}}

\def\cU{{\mathcal U}}
\def\cV{{\mathcal V}}
\def\cR{{\mathcal R}}
\def\cB{{\mathcal B}}
\def\cH{{\mathcal H}}
\def\cE{{\mathcal E}}
\def\cF{{\mathcal F}}
\def\cM{{\mathcal M}}

\def\cP{{\mathcal P}}

\def\cR{{\mathcal R}}

\def\cS{{\mathcal S}}

\def\loc{\operatorname{loc}}

\def\vep{\varepsilon}

\def\TT{{\mathbb T}}
\def\RR{{\mathbb R}}

\title[Diffeomorphisms with mostly expanding center]{Geometrical and measure-theoretic structures of
maps with mostly expanding center}

\author{Jiagang Yang}
\date{\today}
\thanks{J.Y. is partially supported by NNSF 11871487, CNPq, FAPERJ, and PRONEX.}

\address{Department of Mathematics, Southern University of Science and Technology of China, 1088 Xueyuan Rd., Xili, Nanshan District, Shenzhen, Guangdong, China 518055}
\address{Departamento de Geometria, Instituto de Matem\'atica e Estat\'istica, Universidade
Federal Fluminense, Niter\'oi, Brazil}
\email{yangjg\@@impa.br}

\begin{document}

\begin{abstract}
In this paper we study physical measures for $\C^{1+\alpha}$ partially hyperbolic diffeomorphisms with mostly expanding center.
We show that every diffeomorphism with mostly expanding center direction exhibits a geometrical-combinatorial structure, which we call skeleton,
that determines the number, basins and supports of the physical measures. Furthermore, the skeleton allows us to describe how  physical measures bifurcate as the diffeomorphism changes under $C^1$ topology.

Moreover, for each diffeomorphism with mostly expanding center, there exists a $C^1$ neighborhood,
such that diffeomorphism among a $C^1$ residual subset of this neighborhood admits finitely many physical measures, whose basins
have full volume.

We also show that the physical measures for diffeomorphisms with mostly expanding center satisfy exponentially decay of correlation for any H\"older observes.
\end{abstract}

\maketitle

\tableofcontents

\section{Introduction and statement of results}
{\em Physical measures} were introduced in 1970's by Bowen, Ruelle and Sinai to study the large time behavior of Lebesgue typical points for Axiom A attractors. Such systems do not preserve volume (or any measure that is equivalent to the volume) due to the contracting near the attractor. For this reason, those measures are often supported on a zero volume subset of the manifold, but captures the behavior of points in a large set with positive Lebesgue measure. More precisely, an invariant measure $\mu$ is called a physical measure, if the set
$$
\cB(\mu):=\{x\in M: \frac{1}{n}\sum_{i=0}^{n-1}\delta_{f^i(x)}\overset{weak *}{\longrightarrow}\mu\}
$$
has positive volume. This set is known as the {\em basin of $\mu$}. For Axiom A attractors, many properties of physical measures were studied by many different authors. We refer the readers to the review paper~\cite{Y02} and the book~\cite{BDVnonuni} for more details.

It is also known that the physical measures of Axiom A attractors have strong statistical property, one of the most important being the {\em decay of correlations}. It can be seen as the speed at which the system losses dependence and starts to behave like a random system. To be more precise:
\begin{definition}\label{df.decay}
	Given observables $\phi,\psi: M\to \mathbb{R}$, we define the correlation function with
	respect to a measure $\mu$ as
	$$C_\mu(\phi,\psi\circ f^n)=\mid \int \phi(\psi\circ f^n)d\mu -\int \phi d\mu \int \psi d\mu \mid \text{ for } n\geq 1.$$
	We say that the system has decay of correlations, if for all $\phi$ and $\psi$ in some families of functions, $C_\mu(\phi,\psi\circ f^n)$ converges to zero as $n$ goes to infinity.
\end{definition}

With that we are ready to introduce our first result of this paper.

\begin{main}\label{m.a}
Let $f$ be a $C^2$, partially hyperbolic, volume preserving diffeomorphism with one dimensional center. Assume that $f$ is accessible, and that the center Lyapunov exponent of the volume is non-vanishing. Then $f$ has exponential decay of correlations: there is $d>0$ such that
$$
C_{\Leb}(\phi,\psi\circ f^n) = \cO(e^{-dn})
$$
for all H\"older continuous $\phi:M\to\RR$, and $\psi\in L^\infty(\Leb)$. Furthermore, the volume measure is Bernoulli.
\end{main}

This theorem generalizes~\cite[Corollary 0.2]{BW}, where it is shown that every $C^2$, accessible, partially hyperbolic diffeomorphism with one dimensional center is ergodic and has K-property. We remark that such systems are abundant, see the discussion in Section~\ref{ss.3.2}. Also note that the hyperbolicity assumption in the previous theorem (non-vanishing center exponent) was thought rather weak, yet we obtain strong statistical property in the form of fast decay of correlations, central limit theorem and exponential large deviation control (the latter two results are the natural consequences of the decay of correlations; see~\cite{AL} and~\cite{G06}).

By~\cite[Section 8]{Y}, if $f$ is a $C^2$, partially hyperbolic, volume preserving diffeomorphism with one dimensional center and $\lambda^c(\Leb)\ne 0$, then either $f$ or $f^{-1}$ has {\em mostly expanding center}. The rest of this paper is devoted to a general theory on such diffeomorphisms. In particular, Theorem~\ref{m.a} is a direct consequence of Theorem~\ref{main.decay of Correlations} below.
\subsection{Diffeomorphisms with mostly expanding center}

Shortly after the physical measures were introduced for Axiom A attractors, a program for investigating the physical measures of diffeomorphisms beyond uniform hyperbolicity was initiated by Alves, Bonatti, Viana in a sequence of papers, such as \cite{ABV00,BV00} to name but a few.
They introduced several classes of systems, for which the physical measures exist, and the number of physical measures is finite. Among them are the {\em diffeomorphisms with mostly contracting center}, and {\em diffeomorphisms with mostly expanding center}. In this paper, we are particularly interested in the latter class.

{\em Diffeomorphisms with mostly expanding center} are, roughly speaking, partially hyperbolic diffeomorphisms whose center Lyapunov exponents are positive. This class of systems was introduced by
Alves, Bonatti and Viana (\cite{ABV00}) using a different, more technical definition.  Later, another definition was given
by Dolgopyat \cite{D2}, and more recently by Andersson and V\'asquez \cite{AC}. In \cite{AC}, they
also proposed the latter, somewhat stronger, definition as the official definition of having
mostly expanding center, which we will follow in this paper.

We call a diffeomorphism $f$ \emph{partially hyperbolic}, if there exists a
decomposition $TM = E^s \oplus E^c \oplus E^u$ of the tangent bundle $TM$ into three
continuous invariant sub-bundles: $E^s_x$ and $E^c_x$ and $E^u_x$, such that $Df \mid E^s$ is
uniform contraction, $Df\mid E^u$ is uniform expansion and $Df \mid E^c$ lies in between them:
$$
\frac{\|Df(x)v^s\|}{\|Df(x)v^c\|} \le \frac 12
\quand
\frac{\|Df(x)v^c\|}{\|Df(x)v^u\|} \le \frac 12
$$
for any unit vectors $v^s\in E^s_x$, $v^c\in E^c_x$, $v^u\in E^u_x$ and any $x\in M$.
This notation was proposed by Brin, Pesin~\cite{BP} and Pugh,
Shub~\cite{PS72} independently as early as 1970's.

As shown by Bonatti, D\'{i}az and Viana \cite{BDVnonuni} and Dolgopyat \cite{D}, physical measures of any $C^{1+\alpha}$ partially hyperbolic diffeomorphism should be a {\em Gibbs
	$u$-state}, meaning that the conditional
measures of $\mu$ with respect to the partition into local strong-unstable manifolds are absolutely
continuous with respect to Lebesgue measure along the unstable leaves.

\begin{definition}\label{df.me}
A partially hyperbolic diffeomorphism $f: M\to M$ is \emph{mostly expanding along the central direction} if $f$ has
positive central Lyapunov exponents almost everywhere with respect to every Gibbs $u$-state
for $f$.
\end{definition}
This definition is comparable to diffeomorphisms with {\em mostly contracting center} (see, for example,~\cite{DVY}), and share similar properties with the latter. In particular,
$C^1$ openness of the partially hyperbolic
diffeomorphisms with mostly expanding center was recently proved in \cite{Y}. Note however, that the inverse of a diffeomorphism with mostly expanding center may not be mostly contracting. This is because the space of Gibbs $u$-states of $f$ could be very different from that of $f^{-1}$.

A list of examples for partially hyperbolic diffeomorphisms with mostly expanding center will be provided in Section~\ref{s.exME}.

\subsection{Index-$\dim(E^{cu})$ skeleton}
In this article, we will introduce a topological structure of $f$, known as the {\em skeleton}, and use it to study the structure of physical measures of $f$. To this end, for a $C^1$ partially hyperbolic diffeomorphism $f$ with partially hyperbolic splitting $E^{s}\oplus E^{c}\oplus E^u$, we denote by
$i_{cu}=\dim(E^{cu})$ and $i_s=\dim(E^s)$, where $E^{cu}= E^c\oplus E^u$.

\begin{definition}\label{d.1}\footnote{In \cite{DVY}, a different type of skeleton
was defined for diffeomorphisms with mostly contracting center, where the index of saddles in $\cS$ equals $i_{cs}=\dim(E^s\oplus E^c)$. Instead of condition (a),
there the union of stable manifold of periodic orbits of the skeleton is a $u$-section. The existence of index $i_{cs}$ skeleton is a $C^1$ open
property, but it is not necessarily true any more for index $i_s$ skeleton. For more discussions, see Section~\ref{s.skeleton}.}
We say that $\cS$ is an \emph{index $i_{s}$ skeleton of $f$} if $\cS=\{p_1,\cdots, p_k\}$
consists of finitely many hyperbolic saddles with stable index $i_s$, such that:
\begin{itemize}
\item[(a)] $\bigcup_{i=1,\cdots k}\cF^s(\Orb(p_i))$ is dense in $M$;
\item[(b)] $\cS$ does not have a proper subset that satisfies property (a).
\end{itemize}
A set $\cS$ consisting of finitely many hyperbolic saddles with stable index $i_s$ and satisfying (a) above is called a \emph{pre-skeleton}.
\end{definition}

Let us observe that in general, a partially hyperbolic diffeomorphism may not have any skeleton, since
it may not have any hyperbolic periodic orbit at all. Even if it admits a set of periodic points such that the union of their
stable manifolds are dense, such set may have infinite cardinality. However, we will see in Section~\ref{s.skeleton} that if $f$ does have a skeleton, then all the skeletons of $f$ (with the same index) must have the same cardinality (Lemma~\ref{l.skeletondifferent}). Furthermore, every pre-skeleton of $f$ contains a skeleton (Lemma~\ref{l.subsetofpre}).

Finally, in Proposition~\ref{l.existenceskeleton} we will show that if $f$ is $C^{1+\alpha}$ with mostly expanding center (or if $f$ is $C^1$ and close to a $C^{1+\alpha}$ diffeomorphism with mostly contracting center), then $f$ has an index $i_s$ skeleton. Furthermore, in Section~\ref{s.robustskeleton} we will see that the skeletons are robust under $C^1$ topology, in the sense that the continuation of a skeleton of $f$ is a pre-skeleton for nearby $C^1$ maps. Note however, that this property requires $f$ to have mostly expanding center, unlike those skeletons in~\cite{DVY}.

The main result of this paper shows that for such diffeomorphisms,
skeletons  provide rich geometrical information on the physical measures of $f$.

For simplicity, we will frequently suppress the dependence on the H\"older index $\alpha$ and write $C^{1+}$, as the H\"older index $\alpha$ does not play any particular role.
\begin{main}\label{main.sksleton}
Let $f$ be a $C^{1+}$ diffeomorphism with mostly expanding center.
Then $f$ admits an index $i_s$ skeleton. Moreover, Let
$\cS=\{p_1,\cdots,p_k\}$ be any index $i_s$ skeleton of $f$, then for
each $p_i\in \cS$ there exists a distinct physical measure $\mu_i$ such that:
\begin{itemize}
\item[(1)] both the closure of $W^u(Orb(p_i))$ and the homoclinic class of the orbit $\Orb(p_i)$
 coincide with $\supp(\mu_i)$;
\item[(2)] the closure of $\cF^s(\Orb(p_i))$ coincides with the closure of the basin
of the measure $\mu_i$.
\end{itemize}
In particular, the number of physical measures of $f$ is precisely $k=\#\cS$. Moreover,
$$\Int(\Cl(\cB(\mu_i)))\cap \Int(\Cl(\cB(\mu_j)))= \emptyset$$
for $1 \leq i \neq j \leq k$, where $\cB(\mu_i)$ is the basin of $\mu_i$.
\end{main}

\begin{remark}\label{rk.realbasinlocation}
From the proof of Theorem~\ref{main.sksleton}, we have more detailed description on the basins of $\mu_i$: for every $p_i\in \cS$,
denote  by
$$\cO_i=\bigcup_{x\in W^u(\Orb(p_i))} \cF^s(x),$$
then $\cO_i$ contains an open neighborhood of $\Orb(p_i)$. We are going to show that
$\cO_i$ is open and dense in $\Cl(\cF^s(\Orb(p_i)))=\Cl(\cB(\mu_i)))$. Moreover, $\cB(\mu_i)$ is a full volume subset of $\cO_i$,
and $\cO_i\cap \cO_j=\emptyset$ for $1\leq i\neq j\leq k$. This shows that the basin of different physical measures are {\em topologically} separated.
\end{remark}

We would like to mention that the idea of using homoclinic classes to study measures was initiated by \cite{HHTU}, see also
\cite{DVY} and \cite{BCS} for recent similar results.

As a corollary of the previous theorem, we are going to show that any iteration of $f$  still has mostly expanding center; furthermore, the number of physical measures of $f^k$ is also determined by the skeleton of $f$:

\begin{maincor}\label{maincor.period}
Let $f$ be a $C^{1+}$ partially diffeomorphism with mostly expanding center, and
$\cS=\{p_1,\cdots,p_k\}$ be any index $i_s$ skeleton of $f$. Then for any $n>0$, $f^n$ has
mostly expanding center, and  has finitely many physical measures with number bounded by
\begin{equation}\label{eq.P}
P=\prod_{i=1}^k \pi(p_i), \text{ where $\pi(p_i)$ denotes the period of $p_i$.}
\end{equation}
Moreover, every physical measure of $f^P$ is Bernoulli.
\end{maincor}

\subsection{Perturbation of physical measures}
It was shown in \cite{Y} that partially hyperbolic diffeomorphisms with mostly expanding center are
$C^1$ open, i.e., if a $C^{1+}$ diffeomorphism $f$ has mostly expanding center, then any
$C^{1+}$ diffeomorphism $g$ which is sufficiently $C^1$ close to $f$  also has mostly expanding center.
In the following we will analyse how the physical measures vary with respect to the $C^{1+}$ diffeomorphisms
in $C^1$ topology, which generalizes a similar result of Andersson and V\'asquez (\cite{AC2}) under $C^{1+\alpha}$ topology. The key observation here is that physical measures of $f$ are associated with skeletons, which behaves well under $C^1$ topology.

\begin{main}\label{main.robust}
Let $f:M\to M$ be a $C^{1+}$ partially hyperbolic diffeomorphism with mostly expanding center.
Then there exists a $C^1$ neighborhood $\cU$ of $f$ such that the number of physical measures
depends upper semi-continuously in $C^1$ topology among diffeomorphisms in $\Diff^{1+}(M)\cap \cU$.
Moreover, the number of physical measures is locally constant and
the physical measures vary continuously in the weak* topology on an $C^1$ open
and dense subset $\cU^\circ \subset \cU$.
\end{main}

Indeed, the skeletons of  $f$ provide even more  information on the physical measures
for $C^1$ perturbed $C^{1+}$ diffeomorphisms. In particular, the skeletons allow us to describe how the
physical measures bifurcate as the diffeomorphism changes. To this end, we write $p_i(g)$ the continuation of  the hyperbolic saddle $p_i$ for $g$ in a $C^1$ neighborhood of $f$. Theorem~\ref{main.robust} is a direct consequence
of the following, more technical result:

\begin{main}\label{main.robustskeleton}
Let $f$ be a $C^{1+}$ partially hyperbolic diffeomorphism with mostly expanding center,
and $\cS=\{p_1,\cdots, p_k\}$ be a skeleton of $f$. There exists a $C^1$ neighborhood $\cU$
of $f$ such that, for any $C^1$ diffeomorphism $g\in \cU$, there is a subset of $\cS(g) =\{p_1(g),\cdots, p_k(g)\}$ which is a skeleton.
Consequently, for $g\in \Diff^{1+}(M)\cap \cU$, the number of physical measures of $g$ is
no larger than the number of physical measures of $f$. Moreover, these two numbers coincide
if and only if there is no heteroclinic intersection within $\{p_i(g)\}$. In this case, each physical measure of $g$ is close to some physical measure
of $f$, in the weak-* topology.

In addition, restricted to any subset of $\cV\subset \cU$ where the number of physical measures
is constant, the supports of the physical measures
and the closures of their basins vary in a lower semi-continuous fashion,
 in the sense of the Hausdorff topology.
\end{main}

\subsection{Existence of physical measures for $C^1$ generic diffeomorphisms}
Previously, the study of physical measures is mainly focused on maps that are  sufficiently smooth, i.e., with $C^{1+}$ regularity. Recently, the new technique developed in~\cite{HYY, CYZ} enables us to shows the existence of physical measure for a large family of $C^1$ diffeomorphisms, such as those with mostly contacting center.

In this paper, we will  further show the existence of physical measures for $C^1$ generic diffeomorphisms close
to a partially hyperbolic diffeomorphism $f$ that has mostly expanding center. 

Before stating the main theorem of this section, we need the following definition:
\begin{definition}\label{df.Lyapunovstable}
A set $\Lambda$ of a homeomorphism $f$ is \emph{Lyapunov stable} if there is a
sequence of open neighborhoods $U_1\supset U_2\supset \cdots$ such that:
\begin{itemize}
\item[(a)] $\bigcap U_i=\Lambda$;
\item[(b)]$f^n(U_{i+1})\subset U_i$ for any $n,i\geq 1$.
\end{itemize}
\end{definition}

A set being Lyapunov stable means that points starting near $\Lambda$ will not travel too far away from this set under  forward iterations of $f$. However, this does not mean that $\Lambda$ is an attractor.

We have the following $C^1$ locally generic result, which generalizes Theorem~\ref{main.robustskeleton}. We state it as a standalone result since the techniques involved are quite different from Theorem~\ref{main.robustskeleton}.

\begin{main}\label{main.generic}
Let $f:M\to M$ be a $C^{1+}$ partially hyperbolic diffeomorphism with mostly expanding center,
and $\cS=\{p_1,\cdots, p_k\}$ be a skeleton of $f$. Then there exists a $C^1$ neighborhood $\cU$ of
$f$ and a $C^1$ residual subset $\cR\subset \cU$, such that  every $C^1$ diffeomorphism $g\in \cR$ admits finitely
many physical measures, whose basins have full volume. The number of physical measures of $g$ coincides with
the cardinality  of its skeleton, which is no more than the number of physical measures of $f$.
Moreover, the physical measures of $g$ are supported on disjoint Lyapunov stable chain recurrent classes, each of which  is the homoclinic class of some saddle in its skeleton.
\end{main}

\subsection{Statistical properties}

To study the speed of decay of correlations for systems beyond uniformly hyperbolic, in \cite{You98}
Young used a type of Markov partitions with infinitely many symbols to build towers for
systems with non-uniform hyperbolic behavior. These structures are nowadays
commonly referred to as Gibbs-Markov-Young (GMY) structures (see for instance \cite{AL}.)
And it is well known that such maps have exponential speed of decay of correlations  whenever the GMY
structure has exponentially small tails. By Alves and Li in \cite{AL} which is built on the work of Gou\"ezel~\cite{G06},
the latter case happens if the center bundle has certain expansion and
moreover, the tail of {\em hyperbolic times} is exponentially small.

We are going to show that Alves and Li's criterion can be applied to partially hyperbolic diffeomorphisms with mostly expanding
center, and in particular, we prove exponential decay of correlations and exponential
large deviations for the physical measures of $f$, provided that $f$ has mostly expanding center.

\begin{main}\label{main.decay of Correlations}
Let $f:M\to M$ be a $C^{1+}$ partially hyperbolic diffeomorphism with mostly expanding center, $\cS=\{p_1,\cdots,p_k\}$
be a skeleton of $f$ and $P=\prod_{i=1}^k \pi(p_i)$. Then for every physical measure $\mu$ of $f^P$, there is $d>0$ such that
$$C_\mu(\phi,\psi\circ f^{Pn})=\cO(e^{-dn})$$
for H\"older continuous $\phi: M\to \mathbb{R}$, and $\psi\in L^\infty(\mu)$.
\end{main}

\begin{maincor}\label{maincor.large deviations}
Under the assumptions of Theorem~\ref{main.decay of Correlations}, for every physical measure  $\mu$ of $f^P$  and any  H\"older continuous function $\phi$,
the limit exists:
$$\sigma^2=\lim_{n\to \infty} \frac{1}{n}\int (\sum_{j=0}^{n-1}\phi\circ f^{jP} -n\int \phi d\mu)^2d\mu.$$
Moreover, if $\sigma^2>0$, then there is a rate function $c(\vep)>0$ such that
$$\lim_{n\to \infty} \frac{1}{n}\log \mu(\mid \sum_{j=0}^{n-1}\phi\circ f^{jP} -n\int \phi d\mu \mid \geq \vep)=-c(\vep).$$
\end{maincor}

\subsection{Robustly transitive partially hyperbolic diffeomorphisms}
The diffeomorphisms with mostly expanding center also provide a new mechanism to describe the topological
transitivity property. To make this article more complete, we collect two results from two other papers without giving their proof. For more details, see the related papers and the references therein.

\begin{main}\cite{Y}\label{main.transitive}
Let $f$ be a $C^{1+}$ volume preserving, partially hyperbolic diffeomorphism with one dimensional center.
Suppose $f$ is accessible and the center exponent is not vanishing, then $f$ is $C^1$ robustly transitive, i.e.
every diffeomorphism $g$ is transitive for $g$ in a $C^1$ neighborhood $f$ which is not necessarily volume preserving.
\end{main}

\begin{main}\cite{UVY}\label{main.minimalfoliation}
Let $f$ be a $C^{1+}$ partially hyperbolic diffeomorphism with mostly expanding center, such that the stable foliation $\cF^s$
is minimal. Then there is a $C^1$ neighborhood $\cU$ of $f$, such that the stable foliation of any $g\in \cU$ is minimal.
\end{main}

\subsection{Structure of the paper}
This paper is organized in the following way: In  Section~\ref{s.intro} we introduce the main tool of this paper: a special space of probability measures, denoted by $\G(f)$, which is defined using the partial entropy along unstable leaves. This space will serve as the candidate space of physical measures.

In Section~\ref{s.skeleton}, we provide some geometrical properties of skeletons, assuming that such structure exist (which will not be proven until Section~\ref{s.6}). In particular, we will show that every skeleton of $f$ must have the same cardinality, and provide a useful criterion for the existence of a skeleton to be used in the later sections.

Section~\ref{s.5} consists of a direct proof on the existence of physical measures for $C^{1+}$ diffeomorphisms with mostly expanding center. More importantly, we show that the space $\G(f)$ is a finite dimensional simplex that varies upper semi-continuously with respect to the diffeomorphism in $C^1$ topology; moreover, every extreme point of $\G(f)$ is an ergodic physical measure of $f$.

The proof of Theorem~\ref{main.sksleton} and~\ref{main.robustskeleton} occupies the next two sections. We will carefully analyse the non-uniform expanding of $f$ along $E^c$ using hyperbolic times, and use the shadowing lemma of Liao to show the existence of skeletons. We then build a one-to-one correspondence between elements of a skeleton and the physical measures of $f$, and show that physical measures bifurcate as heteroclinic intersections are created between different elements of a skeleton.
Then in Section~\ref{s.robustskeleton}, we generalize the result of Theorem~\ref{main.robustskeleton} to generic $C^1$ diffeomorphisms near $f$.

Section~\ref{s.exME} contains all the existing examples of diffeomorphisms with mostly expanding center, as far as the author is aware. In particular, we collect some very recent examples from~\cite{Y}.

\subsection{On the regularity assumption}\label{s.1.6}

Throughout this article, the regularity assumption on $f$ is changed several times between $C^1$ and $C^{1+}$. For the convenience of the readers, we summarize those changes below:
\begin{enumerate}
	\item having mostly contracting center requires the diffeomorphism to be $C^{1+};$ as a result, the initial diffeomorphism $f$ is always assumed to be $C^{1+}$;
	\item the topology is always $C^1$. Throughout this article, $\cU$ is a neighborhood of $f$ under $C^1$ topology;
	\item the geometrical properties of skeletons only require the diffeomorphism to be $C^1$; this involves Section~\ref{s.skeleton}, Section~\ref{ss.skeleton} and certain part of Section~\ref{s.robustskeleton};
	\item the physical measure having absolutely continuous conditional measure on the unstable leaves and the stable holonomy being absolutely continuous requires $C^{1+}$ regularity, as shown in the classical theory of physical measures. This affects Section~\ref{s.5}, Section~\ref{ss.skeletonandmeasures}, certain part of Section~\ref{s.robustskeleton} and Section~\ref{s.9};
	\item Section~\ref{s.8} deals with $C^1$ generic diffeomorphisms in $\cU$, thus only requires $C^1$ smoothness.
\end{enumerate}

\section{Preliminary\label{s.intro}}

In this section, we introduce some necessary notations and results which will be used later.
Throughout this section, we assume $f$ to be a partially hyperbolic diffeomorphism on the manifold $M$, and
$\mu$ an invariant probability measure of $f$. In Section~\ref{ss.introGibbs} we will assume $f$ to be $C^{1+}$ for the discussion on the Gibbs $u$-states. In Section~\ref{ss.2} and~\ref{ss.candidates}, $f$ is assumed to be $C^1$ only.

\subsection{Gibbs $u$-states\label{ss.introGibbs}}
Following Pesin and Sinai~\cite{PS82} and Bonatti and Viana~\cite{BV00} (see also~\cite[Chapter 11]{BDVnonuni}),
we call \emph{Gibbs $u$-state} any invariant probability measure whose conditional probabilities
(Rokhlin~\cite{Rok49}) along strong unstable leaves are absolutely continuous with respect to the
Lebesgue measure on the leaves. In fact, assuming the derivative $Df$ is H\"older continuous, the
Gibbs-$u$ state always exists, and the densities with respect to Lebesgue measures along unstable
plaques are continuous. Moreover, the densities vary continuously with respect to the strong unstable
leaves. As a consequence, the space of Gibbs $u$-states of $f$,
denoted by $\Gibb^u(\cdot)$, is compact relative to the weak-* topology in the probability space.

The set of Gibbs $u$-states plays important roles in the study of physical measures for
partially hyperbolic diffeomorphisms. The proofs for the following
basic properties of Gibbs $u$-states can be found in the book of Bonatti, D\'{i}az and Viana~\cite[Subsection 11.2]{BDVnonuni} (see also Dolgopyat~\cite{D}):
\begin{proposition}\label{p.Gibbsustates}
Suppose $f$ is a $C^{1+}$ partially hyperbolic diffeomorphism, then
\begin{itemize}
\item[(1)] $\Gibb^u(f)$ is non-empty, weak* compact and convex. Ergodic components of Gibbs $u$-states are Gibbs u-states.
\item[(2)] The support of every Gibbs $u$-state is $\cF^u$-saturated, that is, it consists of entire strong
    unstable leaves.
\item[(3)] For Lebesgue almost every point $x$ in any disk inside some strong unstable leaf, every accumulation point of
    $\frac{1}{n}\sum_{j=0}^{n-1}\delta_{f^j(x)}$ is a Gibbs $u$-state.
\item[(4)] Every physical measure of $f$ is a Gibbs $u$-state; conversely, every ergodic Gibbs $u$-state whose center Lyapunov
    exponents are negative is a physical measure.
\end{itemize}
\end{proposition}

The semi-continuity of Gibbs $u$-states with respect to $C^{1+}$ diffeomorphisms under $C^1$ topology was recently proved by the author of this article in \cite{Y}:

\begin{proposition}\label{p.GibbsuC1regularity}
Suppose $f_n$ ($n=1,\cdots, \infty$) and $f$ are $C^{1+}$ partially hyperbolic diffeomorphisms such that
$f_n\overset{C^1}{\to} f$. Then
$$\limsup \Gibb^u(f_n)\subset \Gibb^u(f),$$
where the convergence is in the Hausdorff topology of the probability space.
\end{proposition}

The following lemma shows the relation between the Gibbs $u$-states of a diffeomorphism and its iterations.

\begin{lemma}\label{l.power}
For any $n>0$, $\Gibb^u(f)\subset \Gibb^u(f^n)$. conversely, let $\nu$ be any Gibbs $u$-state of $f^n$,
then $\frac{1}{n}\sum_{i=0}^{n-1}f^i(\nu)$ is a Gibbs $u$-state of $f$.
\end{lemma}
\begin{proof}
Let $\mu$ be a Gibbs $u$-state of $f$, then it is also an invariant probability of $f^n$. Since $f$ and $f^n$ share the same
unstable foliation, $\mu$ must have the same disintegration along the unstable plaques. Then it follows from the
definition that $\mu$ is also a Gibbs $u$-state of $f^n$.

On the other hand, it is clear that $\frac{1}{n}\sum_{i=0}^{n-1}f^i(\nu)$ is an invariant probability of $f$. By a similar argument as
above, $\frac{1}{n}\sum_{i=0}^{n-1}f^i(\nu)$ is a Gibbs $u$-state of $f$.
\end{proof}

\subsection{Partial entropy along unstable foliation}\label{ss.2}
In this section, we give the precise definition of the partial metric entropy of
$\mu$ along the unstable foliation $\cF^u$ of $f$, which depends on a special class
of measurable partitions. The partial entropy has been proven to be a powerful tool in the study of partially hyperbolic diffeomorphisms, thanks to its semi-continuity in the $C^1$ topology
(\cite{Y}).

\begin{definition}
We say that a measurable partition $\xi$ of $M$ is \emph{$\mu$-subordinate} to the $\cF$-foliation if
for $\mu$-a.e. $x$, we have
\begin{itemize}
\item[(1)] $\xi(x)\subset \cF(x)$ and $\xi(x)$ has uniformly small diameter inside $\cF(x)$;
\item[(2)] $\xi(x)$ contains an open neighborhood of $x$ inside the leaf $\cF(x)$;
\item[(3)] $\xi$ is an increasing partition, meaning that $\xi \prec f\xi$.
\end{itemize}

\end{definition}

Ledrappier, Strelcyn~\cite{LS82} proved that the Pesin unstable lamination admits some $\mu$-subordinate
measurable partition. The following result is contained in Lemma~3.1.2 of Ledrappier, Young~\cite{LY85a}:

\begin{lemma}\label{l.definitionleafentropy}
For any measurable partitions $\xi_1$ and $\xi_2$ that are $\mu$-subordinate to $\cF$,
we have $h_\mu(f,\xi_1)=h_\mu(f,\xi_2)$.
\end{lemma}

This allows us to define the partial entropy of $\mu$ using any $\mu$-subordinate partition:

\begin{definition}\label{d.partialentropy}
For a $C^1$ partially hyperbolic diffeomorphism $f$ and an invariant measure $\mu$, the \emph{partial $\mu$-entropy along unstable foliation $\cF^u$}, which we denote by $h_\mu(f,\cF^u)$, is defined to be $h_\mu(f,\xi)$
for any $\mu$-subordinate partition $\xi$.
\end{definition}

\begin{proposition}\cite{Y}\label{p.partialupper}
The partial entropy $h_\mu(f,\cF^u)$ varies  upper semi-continuously  with respect to the measures and maps in $C^1$ topology.
\end{proposition}

Although partially entropies are well defined for $C^1$ diffeomorphisms and behaves well under $C^1$ topology, one still need higher regularity such as $C^2$ or at least $C^{1+}$ in order to relate it with other quantities such as Lyapunov exponents or Gibbs $u$-states. The following upper bound for the partial entropy along the unstable foliation $\cF^u$ follows \cite{LY85a, LY85b}.

\begin{proposition}\label{p.Ruelle}
Let $f$ be $C^{1+}$ and $\mu$ be an invariant probability measure of $f$, then
$$h_\mu(f,\cF^u)\leq\int \log \Jac^u(x) d\mu(x).$$
Moreover,
\begin{equation}\label{e.Pesinformula}
h_\mu(f,\cF^u)=\int \log \Jac^u(x) d\mu(x).
\end{equation}
if and only if $\mu$ is a Gibbs $u$-state of $f$.
\end{proposition}
\begin{proof}
The inequality follows by~\cite[Theorem $C^\prime$]{LY85b}, when $f$ is $\C^2$.
It was pointed out by \cite{Br} that the same inequality goes well for $\C^{1+}$
diffeomorphism.

The second part was stated in \cite[Theorem 3.4]{L84}.

\end{proof}
The following equality was built in~\cite[Proposition 5.1]{LY85b}, when $f$ is $\C^2$.
As explained above, it also holds under general situation assuming only $C^{1+}$:

\begin{proposition}\label{p.vanishingcenterexponents}

Let $\mu$ be a probability measure of $f$ such that all the center exponents of $\mu$ are
non-positive, then
$$h_\mu(f,\cF^u)=h_\mu(f).$$

\end{proposition}

\subsection{Other invariant measure subspaces\label{ss.candidates}}
Proposition~\ref{p.Gibbsustates} (4) states that when $f$ is $C^{1+}$, Gibbs $u$-states are the natural candidates of the physical measures of $f$. However, this statement falls apart when $f$ is only $C^1$. This is due to the lack of Pesin's formula (\eqref{e.Pesinformula}, Proposition~\ref{p.Ruelle}) for $C^1$ diffeomorphisms. To solve this issue, we will introduce two candidate spaces of physical measures for such $f$. See~\cite{HYY},~\cite{CYZ} and~\cite{CCE} for their properties.

\begin{definition}\label{df.G}
We define:
\begin{itemize}
\item[(A1)] \begin{equation*}\label{eq.Gibbsu}
\G^u(f)=\{\mu\in \cM_{\inv}(f): h_\mu(f,\cF^u)\geq \int \log(\det(Df\mid_{E^u(x)}))d\mu(x)\};
\end{equation*}

\item[(A2)] \begin{equation*}
\G^{cu}(f)=\{\mu\in \cM_{\inv}(f): h_\mu(f)\geq \int \log(\det(Df\mid_{E^{cu}(x)}))d\mu(x)\}
\end{equation*}
where $E^{cu}=E^c\oplus E^u$.
\end{itemize}
\end{definition}

We denote by
$$\G(f)=\G^{u}(f)\cap \G^{cu}(f).$$

\begin{remark}\label{rk.Gu}
\begin{itemize}
\item[(a)] When $f$ is $C^{1+}$, by Ledrappier~\cite{L84}, $\G^u(f)=\Gibbs^u(f)$.
\item[(b)] By the Ruelle's inequality for partial entropy (see for instance \cite{WWZ}), one can replace the inequality in the definition of $\G^u$ by equality:
$$\G^u(f)=\{\mu\in \cM_{\inv}(f): h_\mu(f,\cF^u)= \int \log(\det(Df\mid_{E^u(x)}))d\mu(x)\}.$$
However, the definition of $\G^{cu}$ remains unchanged due to the possibility of having negative Lyapunov exponents in $E^c$.
\end{itemize}
\end{remark}

We first observe that the spaces above are non-empty; moreover, the space $\G(f)$ contains all the candidates of
physical measures.

\begin{proposition}\label{p.physical}
There is a full volume subset $\Gamma$ such that for any $x\in\Gamma$, any limit of the
sequence $\frac{1}{n}\sum_{i=0}^{n-1}\delta_{f^i(x)}$ belongs to $\G(f)$.
\end{proposition}
\begin{proof}
By \cite{CCE}, for $x$ belonging to a full volume subset, any limit of the sequence $\frac{1}{n}\sum_{i=0}^{n-1}\delta_{f^i(x)}$
belongs to $G^{cu}$. Moreover, by \cite{CYZ,HYY}, for $x$ belonging to a full volume subset, any limit of the sequence
$\frac{1}{n}\sum_{i=0}^{n-1}\delta_{f^i(x)}$ belongs to $G^{u}$. We conclude the proof by taking the intersection
of the two full volume subsets.
\end{proof}

The following property shows that $\G^u(\cdot)$ shares similar properties with $\Gibb^u(\cdot)$ (Proposition~\ref{p.Gibbsustates}).

\begin{proposition}\cite{HYY}[Propositions 3.1, 3.5]\label{p.Gu}
The space $\G^u(f)$ is convex, compact, and varies in a upper semi-continuous way with respect to the partially
hyperbolic diffeomorphisms under $C^1$ topology. Moreover, for any invariant measure $\mu\in \G^u(f)$, every ergodic component of
its ergodic decomposition still belongs to $\G^u(f)$.
\end{proposition}

We need to observe that, in general, the space $\G(f)$ may not have such properties (especially when it comes to the ergodic components). Indeed, in Proposition~\ref{p.existencephysical}, we will show that the above properties holds for $\G(g)$, when $g$ is $C^1$ close to $f$ which is $C^{1+}$ with mostly expanding center.

\section{Examples of partially hyperbolic diffeomorphisms with mostly expanding center \label{s.exME}}

For a long time (before \cite{Y}), there are only two known examples of diffeomorphisms with mostly expanding center
(under the definition that is used in this paper, which is stronger than that in~\cite{ABV00}). These examples  are due to Ma{\~{n}}\'{e}~\cite{Ma} (see~\cite{ABV00} and~\cite[Section 6]{AC})
and Dolgopyat \cite{D2}. We list these examples below, as well as some new examples provided in~\cite{Y}.
Let us recall that the set of partially hyperbolic diffeomorphisms with mostly expanding center is $C^1$
open among $\Diff^{1+}(M)$.

\subsection{Derived from Anosov diffeomorphisms}
We assume $A$ to be a linear Anosov diffeomorphism over $\TT^3$ with $3$
positive simple real eigenvalues $0<k_1<1<k_2<k_3$.
\subsubsection{Local derived from Anosov diffeomorphisms}
Let us begin by recalling the construction of Ma\~n\'e's example, which is a local $C^0$ perturbation of $A$.
The  statement below is a little different from the original construction in the history:

\begin{example}\label{ex.Mane}
Let $p$ be a fixed point of $A$ and $U$ a small neighborhood of $p$. There is a partially hyperbolic diffeomorphism
$f_0$ that coincides with $A$ on $\TT^3\setminus U$. $f_0$ is topological Anosov, and
\begin{equation}\label{eq.topAnosov}
\mid Df_0\mid_{E^c(x)} \mid \geq 1
\end{equation}
where the equality holds if and only if $x=p$.
\end{example}

Since $Df_0\mid_{E^c(\cdot)}$ is expanding everywhere except at the point $p$, it is clear that $f_0$ has
mostly expanding center. Thus, by \cite{Y}, $f_0$ admits a $C^1$ neighborhood $\cU$ such that every $C^{1+}$
diffeomorphism belonging to $\cU$ has mostly expanding center.


\subsubsection{Generalized Derived from Anosov diffeomorphisms}
By the topological classification of partially hyperbolic diffeomorphisms which are isotopic to $A$ (\cite{BBI,HP,U}), we call such
diffeomorphisms \emph{derived from Anosov $A$}, and denote by $\cD\cA(A)$. The following example by Shi, Viana and the author of this paper \cite{SVY}
revises the fact that  $C^{1+}$ volume preserving derived from Anosov diffeomorphisms have mostly expanding center whenever
the volume has large metric entropy.

\begin{example}\label{ex.volumeDA}
Let $f\in \cD\cA(A)$ be a $C^{1+}$ volume preserving partially hyperbolic diffeomorphism and $h_{\Leb}(f)>\log k_3$, then $f$
has mostly expanding center.
\end{example}

\subsection{Perturbation of volume preserving partially hyperbolic diffeomorphisms}\label{ss.3.2}

In \cite{D2}, Dolgopyat showed that:
\begin{example}\label{ex.perturbationgeodesic}
Let $X_1$ be the time one map of a hyperbolic geodesic flow on a surface $M$, then for generic $C^\infty$
perturbation $f$ of $X_1$, either $f$ of its inverse $f^{-1}$ has mostly expanding center.
\end{example}

The following result in \cite{Y} allows us to obtain more examples using   $C^1$ perturbation:

\begin{proposition}\label{p.volumeperturbation}

Let $f$ be a $C^{1+}$ volume preserving partially hyperbolic diffeomorphism
with one-dimensional center. Suppose the center exponent of the volume measure is
positive and $f$ is accessible. Then $f$ admits an $\C^1$ open neighborhood, such that every
$\C^{1+}$ diffeomorphism in this neighborhood (not necessarily volume preserving) has mostly expanding center.

\end{proposition}

Proposition~\ref{p.volumeperturbation} contains abundance of systems: by Avila~\cite{AA}, $\C^\infty$
volume preserving diffeomorphisms are $\C^1$ dense. And by Baraviera and Bonatti~\cite{BB},
the volume preserving partially hyperbolic diffeomorphisms with one-dimensional center
and non-vanishing center exponent are $\C^1$ open and dense. Moreover, the subset
of accessible systems is $\C^1$ open and $\C^k$ dense for any $k\geq 1$ among all
partially hyperbolic diffeomorphisms with one-dimensional center direction, due to the work of  Burns,
Rodriguez Hertz, Rodriguez Hertz, Talitskaya and Ures~\cite{BHHTU}; see also Theorem
1.5 in Ni{\c{t}}ic{\u{a}} and T{\"o}r{\"o}k~\cite{NT}.

Indeed the accessibility assumption in the above proposition can be replaced by another hypothesis:

\begin{example}[see \cite{UVY}]\label{ex.inverseofMC}

Let $f$ be a $C^{1+}$ volume preserving partially hyperbolic diffeomorphism
with one-dimensional center. Suppose the center exponent of the volume measure is
positive and $f^{-1}$ has mostly contracting center. Then $f$ admits an $\C^1$ open neighborhood, such that every
$\C^{1+}$ diffeomorphism in this neighborhood has mostly expanding center.
\end{example}

\begin{remark}
The hypothesis that $f^{-1}$ has mostly contracting center is equivalent to the assumption that
$\cF^s$ is minimal.
\end{remark}

The diffeomorphisms with minimal strong stable and unstable foliations are
also quite common; they fill an open and dense subset of volume preserving partially
hyperbolic diffeomorphisms with one-dimensional center and has  compact center leaves. This follows from a conservative version of the results in~\cite{BDU}.

\subsection{Product of diffeomorphisms with mostly expanding center}

It is shown by Ures, Viana and the author of this article in \cite{UVY} that:

\begin{proposition}\label{p.product}
Suppose $f_1$ and $f_2$ are $C^{1+}$ partially hyperbolic diffeomorphisms over manifolds $M_1$ and $M_2$.
Assume that both $f_1$ and $f_2$ have mostly expanding center. Then $f_1\times f_2$ is a partially hyperbolic
diffeomorphism over $M_1\times M_2$ with mostly expanding center. As a result, nearby $C^{1+}$ diffeomorphisms (which may not be products any more) also have mostly expanding center.
\end{proposition}

\section{Properties of skeleton\label{s.skeleton}}
In this section, we introduce several basic properties for skeletons, although the existence of skeletons will be postponed to Section~\ref{s.6}. The main tool in this Section is the Inclination lemma, also known as the $\lambda$-lemma.

To state the properties of skeletons under general situations, throughout this section, we assume $f$ to be a $C^1$ partially
hyperbolic diffeomorphism with partially hyperbolic splitting $E^s\oplus E^c\oplus E^u$, $\cS=\{p_1,\cdots,p_k\}$ is an index $i_{s}$ skeleton of $f$. In particular, we will not assume $f$ to have mostly expanding center. It is also worth noting that, unlike in~\cite{DVY}, we will not discuss the robustness of skeletons under perturbation of $f$ in this section. Such discussion requires $f$ to have mostly expanding center, and is postponed to Section~\ref{s.robustskeleton} (see Lemma~\ref{l.robustskeleton}).

The first three technical lemmas provide geometrical information on the structure of skeleton. The main result in this section is Lemma~\ref{l.skeletondifferent}, which states that every skeleton of $f$ must have the same cardinality. The last two lemma provide useful criterion for skeletons, which will be used multiple times in the later sections.

\begin{lemma}\label{l.skeleton}
\begin{itemize}
\item[(1)] For any $1\leq i \leq k$, $\Cl(\cF^s(\Orb(p_i)))$ has non-empty interior;
\item[(2)] For $1\leq i\neq j \leq k$, there is no heteroclinic intersection between $\Orb(p_i)$ and $\Orb(p_j)$, i.e., $\cF^s(\Orb(p_i))\cap W^u(\Orb(p_j))=\emptyset$;
\item[(3)] $\Int(\Cl(\cF^s(\Orb(p_i))))\cap \Int(\Cl(\cF^s(\Orb(p_j)))) =\emptyset$.
\end{itemize}
\end{lemma}
\begin{proof}
Because $\cS$ is a skeleton, from (a) of the definition of skeleton,
$$\bigcup_{i=1}^k \Cl(\cF^s(\Orb(p_i)))=M.$$
Suppose by contradiction that $\Cl(\cF^s(\Orb(p_i)))$ has empty interior for some $1\leq i \leq k$, then $\bigcup_{j\neq i}\Cl(\cF^s(\Orb(p_j)))=M$.
Thus $\cS\setminus \{p_i\}$ also satisfies (a) of Definition~\ref{d.1}, which contradicts with (b) of Definition~\ref{d.1} and the fact that $\cS$ is a skeleton.
This finishes the proof of (1).

We are ready to prove (2). First by the unstable manifold theorem, $W^u(\Orb(p_j))$ is tangent to  the bundle $E^{cu}$. Thus if the intersection $\cF^s(\Orb(p_i))\cap W^u(\Orb(p_j))$
is not empty, it must be transversal.
By the Inclination lemma, $\Cl(\cF^s(\Orb(p_j)))\subset \Cl(\cF^s(\Orb(p_i)))$, and thus $\cS\setminus \{p_j\}$ is a pre-skeleton,
a contradiction.

To prove (3), we assume by contradiction that there are $1\leq i \neq j \leq k$ such that
$U=\Int(\Cl(\cF^s(\Orb(p_i))))\cap \Int(\Cl(\cF^s(\Orb(p_j)))) \neq \emptyset$. Take $x\in \cF^s_R(\Orb(p_i))\cap U$ for some $R>0$ where $\cF^s_R(\cdot)$ is the disk in $\cF^s(\cdot)$ with radius  $R$ under leaf metric,
then there is $x_n\in \cF^s(\Orb(p_j))\cap U$ such that $x_n\to x$. By the continuity of stable foliation, we have $\cF^s_{2R}(x_n)\to
\cF^s_{2R}(x)$ and thus for $n$ sufficiently large, $\cF^s_{2R}(x_n) \cap W^u(\Orb(p_i))\neq \emptyset$. Because $x_n\in \cF^s(\Orb(p_j))$,
we have $\cF^s(\Orb(p_j))\cap W^u(\Orb(p_i))\neq \emptyset$, which is a heteroclinic intersection between $p_j$ and $p_i$, a contradiction with item (2).
\end{proof}

In the following, instead of using the open set $\Int(\Cl(\cF^s(\Orb(p_i))))$, we are going to consider the set
$\cO_i=\bigcup_{x\in W^u(\Orb(p_i))}\cF^s(x)$. By the transversality between $E^{cu}$ and $E^s$ and continuity of stable foliation,
the set $\cO_i$ is open. In the following we will reveal the relation between these two open sets.

For a hyperbolic saddle $p$, we denote by $H(p,f)$ the homoclinic class of $p$ with respect to the map $f$, that is, the closure of homoclinic intersections between $W^s(\Orb(p))$ and $W^u(\Orb(p))$.
\begin{proposition}\label{p.skeletongeometry}
For every $p_i\in \cS$,
\begin{itemize}
\item[(i)] $\Cl(W^u(\Orb(p_i)))=H(p_i,f)$;
\item[(ii)]
\begin{equation}\label{eq.closurestable}
\Cl(\cF^s(\Orb(p_i)))=\Cl(\cup_{x\in W^u(\Orb(p_i))}\cF^s(x)),
\end{equation}
thus
$\cO_i$ is open and dense in $\Int(\Cl(\cF^s(\Orb(p_i))))$.

\end{itemize}
\end{proposition}
\begin{proof}
We first prove (i). From the definition of homoclinic class, we have
$$\Cl(W^u(\Orb(p_i)))\supset H(p_i,g).$$
Now let us prove the other direction of the inclusion.

By the definition of skeleton, $\bigcup_{j=1,\cdots,k} (\cF^s(\Orb(p_j)))$ is dense in the manifold $M$. Thus
for any $x\in W^u(\Orb(p_i))$, there is $p_j\in \cS$ such that $x\in \Cl(\cF^s(\Orb(p_j))))$. According to (2) of Lemma~\ref{l.skeleton}, there is
no heteroclinic intersection between $\cF^s(\Orb(p_j))$ and $W^{u}(\Orb(p_i))$ when $i\neq j$, thus $i=j$. It
then follows that $\cF^s(\Orb(p_i))$ and $W^{u}(\Orb(p_i))$ have non-trivial intersections arbitrarily close $x$,
meaning that $x\in H(p_i,f)$. This completes the proof of (i).

By the discussion above, we have shown that $\cF^s(\Orb(p_i))\cap W^u(\Orb(p_i))$ is dense inside $W^u(\Orb(p_i))$, thus
$$\Cl(\cF^s(\Orb(p_i)))\supset \Cl(\cup_{x\in W^u(\Orb(p_i))}\cF^s(x)).$$
Meanwhile, because $\Orb(p_i)\subset W^u(\Orb(p_i))$, the inclusion
$$\Cl(\cF^s(\Orb(p_i)))\subset \Cl(\cup_{x\in W^u(\Orb(p_i))}\cF^s(x))$$
is trivially satisfied, and the equality~\eqref{eq.closurestable} follows immediately.
\end{proof}

The next two lemmas show that if one replaces $p_i\in\cS$ by another hyperbolic periodic point $q\in\cO_i$ with index $i_s$, the new set $\cS' = \cS\cup\{q\}\setminus \{p_i\}$ is still an index $i_s$ skeleton; moreover, any skeleton of $f$ can be obtained in this way.

\begin{lemma}\label{l.skeletonhomoclinic}
Let $q$ be an index $i_s$ hyperbolic periodic point, then $q \in \cO_i$
if and only if $q$ and $p_i$ are homoclinic related with each other. Moreover, $\cS^\prime=\{q\}\bigcup \cS\setminus\{p_i\}$ remains an index $i_s$
skeleton.
\end{lemma}
\begin{proof}
If $q$ and $p_i$ are homoclinic related with each other, take $a\in \cF^s(q)\pitchfork W^u(\Orb(p_i))$ and $U$ a neighborhood of $a$
in $W^u(\Orb(p_i))$. By the continuity of stable foliation, $\bigcup_{x\in U}\cF^s(x)$ contains a neighborhood of $q$. Then by Proposition~\ref{p.skeletongeometry},
$q \in \bigcup_{x\in W^u(\Orb(p_i))}\cF^s(x)=\cO_i$.

On the other hand, suppose $q \in \bigcup_{x\in W^u(\Orb(p_i))}\cF^s(x)$, then there exists an intersection point $a\in \cF^s(q)\pitchfork W^u(\Orb(p_i))$. By Proposition~\ref{p.skeletongeometry}[(i)],
$a\in H(p_i,f)$ and thus can be approached by $\cF^s(\Orb(p_i))$. By the continuity of stable foliation,
$\cF^s(\Orb(p_i))\cap W^u(q)\neq \emptyset$. We conclude that $q$ and $p_i$ are homoclinic related.

Now suppose $q$ and $p_i$ are homoclinic related. Then by the Inclination lemma, we have $\Cl(\cF^s(\Orb(q)))=\Cl(\cF^s(\Orb(p_i)))$, which means that
$$\bigcup_{p\in \cS^\prime}\Cl(\cF^s(\Orb(p)))=M.$$
It remains to show that $\cS^\prime$ does not have a proper subset $\cS''$ that satisfies the above equality.

Assume by contradiction that $\cS''$ is such a proper subset of $\cS'$. Because $\cS$
is a skeleton, ${\cS}''$ has to contain $q$, otherwise ${\cS}''$ will be a proper subset of $\cS$, which  contradicts with the fact that $\cS$ is a skeleton.
By the discussion above,  $\tilde{\cS}=\{p_i\}\bigcup {\cS}''\setminus\{q\}$
is a pre-skeleton. However, this is impossible since $\tilde{\cS}$ is a proper subset of $\cS$.
\end{proof}

\begin{lemma}\label{l.skeletondifferent}
Suppose $\cS^\prime=\{q_1,\cdots,q_l\}$ is a skeleton of $f$, then $l=k$, and after reordering,
$q_i$ and $p_i$ are homoclinic related for $i=1,\cdots, k$.
\end{lemma}

\begin{proof}
By Definition~\ref{d.1} (a), for each $q_j\in \cS^\prime$ there is some $p_i\in \cS$ such that $\cF^s(\Orb(p_i))$
approaches $p_i$; thus $W^u(q_j)$ intersects $\cF^s(p_i)$ transversally.

Choose any such $p_i$ (we will see in a second that the choice is unique). The same argument applied on $p_i$ shows that there exists some $q_k\in \cS^\prime$ such that $W^u(p_i)$ intersects
$\cF^s(\Orb(q_k))$ transversally. By the Inclination lemma, there is transverse intersection between $W^u(\Orb(q_j))$ and $\cF^s(\Orb(q_k))$.
By Lemma~\ref{l.skeleton}[(2)], this can only happens if $j=k$. In particular, $p_i$ and $q_j$ are homoclinically related to one another.

Since being homoclinically related is an equivalent relation, and different elements in a skeleton do not have heteroclinic intersections, it follows that the choice of $p_i$ is unique, and the map
$q_j\mapsto p_i$ is injective. Reversing the roles of $\cS'$ and $\cS$, we also get an injective
map $p_i \mapsto q_j$ which, by construction, is the inverse of the previous one. Thus, both
maps are bijective and, in particular, $\#\cS=\#\cS^\prime$. Moreover, after reordering,
$q_i$ and $p_i$ are homoclinic related for $i=1,\cdots, k$.
\end{proof}

The following lemma provides a useful criterion on the existence of skeletons, which will be used in Section~\ref{s.6}.

\begin{lemma}\label{l.subsetofpre}
Any pre-skeleton contains a subset which forms a skeleton.
\end{lemma}
\begin{proof}
Let $\cS^\prime=\{p_1,\cdots, p_l\}$ be a pre-skeleton.
We first define a relation between the elements of $\cS^\prime$:
we say $p_i\prec p_j$ if $W^u(\Orb(p_i))\pitchfork \cF^s(p_j)\neq \emptyset$.
By the Inclination lemma, it is easy to see that $\prec$ is reflexive and transitive: if $p_i\prec p_j$ then
\begin{equation}\label{eq.hetero}
\Cl(\cF^s(\Orb(p_j)))\supset \Cl(\cF^s(\Orb(p_i))).
\end{equation}
Moreover,
if we have $p_i\prec p_j$ and $p_j\prec p_i$, then we say that they belong to the same equivalent class. Two elements
belong to the same equivalent class if and only if they are homoclinic related.

Now in the set of  equivalent classes, $\prec$ induces a partial order. For every maximal equivalent class under this partial order, we pick up an representative element
and then obtain a subset $\cS\subset \cS^\prime$.
By \eqref{eq.hetero}, $\cS$ is clearly a pre-skeleton. Moreover, from the construction,
the elements of $\cS$ have no heteroclinic intersection. Then this lemma is a corollary of the following result:

\begin{lemma}\label{l.preskeleton}
Let $\cS=\{p_1,\cdots,p_k\}$ be a pre-skeleton of $f$ such that there is no heteroclinic intersection
between $\Orb(p_i)$ and $\Orb(p_j)$ for $1\leq i\neq j\leq k$, then $\cS$ is a skeleton.
\end{lemma}

\begin{proof}
We prove by contradiction. Suppose $\cS$ is not a skeleton, then by Definition~\ref{d.1} (b),  it contains a proper subset $\cS''$
which forms a pre-skeleton. After reordering, we may assume $\cS''=\{p_1,\cdots, p_l\}$ where $l<k$.

Then by the definition of skeleton, $\bigcup_{1\leq i \leq l}\cF^s(\Orb(p_i))$ is dense in the manifold $M$. As a result, there is $1\leq i_0\leq l$ such that $\cF^s(\Orb(p_{i_0}))$ approaches $p_k$, and thus
$\cF^s(\Orb(p_{i_0}))\pitchfork W^u(p_k)\neq \emptyset$, which contradicts with the assumption that there is no heteroclinic intersection between elements of $\cS$. The proof is complete.
\end{proof}

\end{proof}
\section{Diffeomorphisms with mostly expanding center revisit}\label{s.5}
Throughout this section, we assume $f$ to be a $C^{1+}$ partially hyperbolic diffeomorphism with mostly expanding center.
To make this paper as self-contained as possible,  we will provide a
direct proof on the existence of physical measures for diffeomorphisms with mostly expanding center. The proof is
different from the original argument in \cite{ABV00} and is useful
for the discussion in later sections.

One of the main difficulties in the study of diffeomorphisms with mostly expanding center lies in the fact that the space $\Gibb^u(f)$ (or $G^u(g)$ for nearby $C^1$ map  $g$) is `too large', in the sense that it contains plenty of ergodic measures that are not physical.\footnote{In comparison, if $f$ has {\em mostly contracting center}, then every ergodic measure in $\Gibb^u(f)$ is a physical measure, and finiteness follows easily. See~\cite{DVY} and~\cite{HYY} for the discussion there.}

We start solving this issue by introducing the following description for diffeomorphisms with mostly expanding center, which turns out to be equivalent to Definition~\ref{df.me}. The main advantage is that  it gives a uniform estimate on the center Lyapunov exponents for measures in $\Gibb^u(f)$.

\begin{proposition}\cite{Y}[Proposition 6.1]\label{p.onestep}
Suppose $f$ has mostly expanding center, then there is $N_0\in \mathbb{N}$ and $b_0>0$ such that,
for any $\tilde{\mu}\in \Gibb^u(f^{N_0})$,
\begin{equation}\label{eq.definitionofme}
\int \log \|Df^{-N_0}|_{E^{cu}(x)}\|d\tilde{\mu}(x)<-b_0.
\end{equation}
\end{proposition}
\begin{remark}\label{rk.step}
From now on, we assume $N_0=1$.
\end{remark}

By the upper semi-continuity of the space $\G^u(f)$ with respect to diffeomorphisms in $C^1$ topology (Proposition~\ref{p.Gu}), we can extend this estimate to nearby $C^1$ maps:
\begin{lemma}\label{l.robustGexpand}
There is a $C^1$ open neighborhood $\cU$ of $f$, such that for any $C^1$ diffeomorphism $g\in \cU$,
and any $\mu\in \G^u(g)$, we have
\begin{equation}\label{eq.uniformme}
\int \log \|Dg^{-1}|_{E^{cu}_g(x)}\|d\mu(x)<-b_0.
\end{equation}
\end{lemma}
This is later used in Section~\ref{ss.hyperbolictime}, where we show that for any $C^1$ diffeomorphism $g$ in a small
$C^1$ neighborhood $\cU$ of $f$, and for any $\mu\in \G^u(g)$, $\mu$ typical points $x$ have infinitely many {\em hyperbolic times} for the bundle
$E^{cu}$ in its orbit (see Lemma~\ref{l.Pliss}).

On the other hand, the space $\G^{cu}(f)$ is also `too large' since it may contain measures with negative center exponents.  Such measures need not be a Gibbs $u$-state, thus not physical due to Proposition~\ref{p.Gibbsustates} (4). One way to solve this issue is to take the space of intersection, $\G(f)$, which is a much smaller space to work with. However, this creates another problem: unlike the partial entropy which is upper semi-continuous (which makes the space $\G^u(f)$ upper semi-continuous in $f$), the metric entropy $h_\mu$ may not have such property. This is dealt with in Section~\ref{ss.fakefoliation}, as we introduce fake foliations for partially hyperbolic diffeomorphisms,
and show in Lemma~\ref{l.uniformentropyexpansiveness} that the measures in $\G^u(g)$ for $g\in \cU$ are uniformly entropy expansive. As a consequence, in Section~\ref{ss.uppermetricentropy} it is shown (Corollary~\ref{c.uppersemicontinuous})
that metric entropy,  {\em when restricted to measures in $\G^u(g)$}, varies in a upper semi-continuous fashion in weak-* topology and
with respect the diffeomorphism $g\in \cU$ in $C^1$ topology.

Finally, Section~~\ref{ss.physicalmeasures} contains the main result of this section:
for any $C^{1+}$ diffeomorphism $g\in \cU$, every extreme element of $\G(g)$ is an ergodic physical measure of $g$.

\subsection{Hyperbolic times\label{ss.hyperbolictime}}
\begin{definition}
Given $b>0$, we say that $n$ is a $b$-\emph{hyperbolic time} for a point $x$ if
$$\frac{1}{k}\sum_{j=n-k+1}^n\log\|Df^{-1}\mid_{E^{cu}(f^j(x))}\|\leq -b \text{ for any $0<k \leq n$}.$$
\end{definition}

Let $D$ be any $\C^1$ disk, we use $d_D(\cdot,\cdot)$ to denotes the distance between two points in
the disk. Recall that for the dominated splitting $E^s\oplus E^{cu}$, one can define the center unstable cone field, which is invariant under forward iteration.

The next lemma states that if $n$ is a hyperbolic time for $x$, then on the disk $f^n(D)$, one picks up an contraction by $e^{-b}$ for each backward iteration.

\begin{lemma}[\cite{ABV00} Lemma 2.7]\label{l.hyperbolictime}
For any $b>0$, there is $r>0$ such that, given any $\C^1$ disk $D$ tangent to the
center-unstable cone field, $x\in D$ and $n\geq 1$ a $b/2$-hyperbolic time for $x$, we have
$$d_{f^{n-k}(D)}(f^{n-k}(y),f^{n-k}(x))\leq e^{-kb/2} d_{f^n(D)}(f^n(x),f^n(y)),$$
for any point $y\in D$ with $d_{f^n(D)}(f^n(x),f^n(y))\leq r$ and any $1\leq k \leq n$.
\end{lemma}

\begin{remark}\label{r.uniformunstable}
For fixed $b_0/2>0$, we can take $r=r_1$ to be constant for the diffeomorphisms
in a $\C^1$ neighborhood of $f$.
\end{remark}

By Lemma~\ref{l.robustGexpand} and Proposition~\ref{p.physical}, for any $g\in \cU$, there is a full volume subset $\Gamma_g$ such that for any $x\in\Gamma_g$, any limit of the
sequence $\frac{1}{n}\sum_{i=0}^{n-1}\delta_{g^i(x)}$ belongs to $\G(g)$. Thus for any $x\in \Gamma_g$,
\begin{equation}
\begin{split}
\limsup_{n\to \infty}\frac{1}{n}\sum_{i=0}^{n-1}\log \|Dg^{-1}\mid_{E^{cu}_g(g^i(x))}\|&=\limsup \int \log \|Dg^{-1}\mid_{E^{cu}_g(x)}\| d\frac{1}{n}\sum_{i=0}^{n-1}\delta_{g^i(x)}\\
&<-b_0<0.
\end{split}
\end{equation}

Define $H(b_0/2,x,g)$ to be the set of $b_0/2$-hyperbolic times for $x\in \Gamma_g$, that is, the set of times $m\geq 1$ such that
\begin{equation}\label{eq.hyperbolictime}
\frac{1}{k}\sum_{i=m-k+1}^{m} \log \|Dg^{-1}\mid_{E^{cu}_g(g^i(x))}\|\leq -b_0/2 \text{ for all $1\leq k \leq m$}.
\end{equation}
By the Pliss Lemma (see \cite{ABV00}), such hyperbolic times have  positive density on the orbit segment from $0$ to $n$: there exists $n_x\geq 1$ and $\delta_1>0$ such that
\begin{equation}\label{eq.Pliss}
\#(H(b_0/2,x,g)\cap [1,n))\geq n\delta_1 \text{ for all $n\geq n_x$ }.
\end{equation}

By Lemma~\ref{l.hyperbolictime} and Remark~\ref{r.uniformunstable}, there is $r_1>0$
which only depends on $\cU$ and $b_0/2$, such that for any $x\in \Gamma_g$, and any disk
$D$ tangent to the center-unstable cone field, $x\in D$, $n\in H(b_0/2,x,g)$, we have
\begin{equation}\label{eq.cuhyperbolic}
d_{D}(x,y)\leq e^{-nb_0/2}d_{g^n(D)}(g^n(x),f^n(y)),
\end{equation}
for any $y\in D$ with $d_{g^n(D)}(g^n(x),g^n(y))\leq r_1$  (We also assume that $r_1$
satisfies the condition \eqref{eq.sizeunstable} below, which depends only on the neighborhood $\cU$.)
In particular, for $x\in \Gamma_g\cap D$,
$g^n(D)$ contains a smaller disk $D_n$ with diameter $r_1$ for $n\in H(b_0/2,x,g)$ sufficiently large.
Then $\cup_{z\in D_n}\cF^s(z)$ contains an open ball with radius $r_1$.

\begin{definition}\label{df.hyperbolicpt}
Denote by $\cH(b_0/2,g)$ the set of point $x$ such that for any $k\geq 1$,
\begin{equation}\label{eq.sethyperbolictime}
\frac{1}{k}\sum_{i=0}^{k-1} \log \|Dg^{-1}\mid_{E^{cu}_g(g^{-i}(x))}\|\leq -b_0/2 \text{ for all $k\geq 0$}.
\end{equation}
In other words, for every $n>0$, $n$ is a hyperbolic time for the point $f^{-n}(x)$.
\end{definition}

The next lemma shows that there are plenty of hyperbolic times on the forward orbit of $x$, every $\mu\in\G^u(g)$ and $\mu$ almost every $x$.

\begin{lemma}\label{l.Pliss}
For any $g\in \cU$ and any $\mu\in \G^u(g)$, we have
\begin{equation}\label{eq.measurehyperbolicitme}
\mu(\cH(b_0/2,g))\geq \delta_1
\end{equation}
where $\delta_1$ is given in \eqref{eq.Pliss}.
\end{lemma}
\begin{proof}
By Proposition~\ref{p.Gu}, we may assume $\mu$ to be ergodic. By Birkhoff theorem, we only need to show that
for $\mu$ almost every $x$, $\liminf\frac{1}{n}\#\{1\leq k \leq n; f^k(x)\in \cH(b_0/2,g)\}\geq \delta_1$.
It is equivalent to show that for some fixed $m_x$,
\begin{equation}\label{eq.averagehyperbolic}
\liminf\frac{1}{n}\#\{1\leq k \leq n; f^{k+m_x}(x)\in \cH(b_0/2,g)\}\geq \delta_1.
\end{equation}

By Lemma~\ref{l.robustGexpand}, take $x$ be a typical point of $\mu$, such that
$$\lim\frac{1}{n}\sum_{i=0}^{n-1} \log \|Dg^{-1}\mid_{E^{cu}_g(g^{-i}(x))}\|\leq -b_0.$$
We claim that there is $m>0$ such that $g^{-m}(x)\in \cH(b_0/2,g)$. Otherwise
for any $g^{-n}(x)$ , there is $i_n>0$ such that
$\frac{1}{i_n}\sum_{0}^{i_n-1} \log \|Dg^{-1}\mid_{E^{cu}_g(g^{-i-n}(x))}\|\geq -b_0/2$.
Recursively, we obtain a sequence of points: $n_1=i_0$, $n_2=n_1+i_{n_1}, \cdots$; by induction, we have
$$\frac{1}{n_k}\sum_{i=0}^{n_k-1} \log \|Dg^{-1}\mid_{E^{cu}_g(g^{-i}(x))}\|\geq -b_0/2.$$
This contradicts with the choice of $x$.

Moreover, it is easy to see that for any $k\in H(b_0/2,g^{-m}(x),g)$,
$g^{k-m}(x)\in \cH(b_0/2,g)$.
Then by \eqref{eq.Pliss} and take $m_x=m$ in \eqref{eq.averagehyperbolic}, we conclude the proof.
\end{proof}
\subsection{Fake foliations\label{ss.fakefoliation}}
In order to avoid  assuming dynamical coherence of $f$, we use locally invariant
({\em fake}) foliations, a construction that follows  Burns, Wilkinson \cite{BW} and goes back to Hirsch, Pugh, Shub \cite{HPS}. We fix $\cU$ a small $C^1$ neighborhood of $f$ provided by Lemma~\ref{l.robustGexpand}.

\begin{lemma}\label{l.fake}
There are real numbers $\rho>r_0>0$ only depending on $\cU$ with the following properties. For any $x\in M$,
the neighborhood $B(x,\rho)$ admits foliations $\hat{\cF}^s_{g,x}$ and $\hat{\cF}^{cu}_{g,x}$ such that
for every $y\in B(x,r_0)$ and $*=\{s,cu\}$:
\begin{itemize}
\item[(1)] the leaf $\hat{\cF}^i_{g,x}(y)$ is $C^1$, and its tangent bundle $T_y(\hat{\cF}^i_{g,x}(y))$ lies in a cone of $E^i(x)$;
\item[(2)] $g(\hat{\cF}^s_{g,x}(y,r_0))\subset \hat{\cF}^s_{g,g(x)}(g(y))$ and $g^{-1}(\hat{\cF}^{cu}_{g,x}(y,r_0))\subset \hat{\cF}^{cu}_{g,f^{-1}(x)}(g^{-1}(y))$;
\item[(3)] we have product structures on the $B(x,r_0)$, i.e., for any $y,z\in B(x,r_0)$, there is a unique intersection
between $\hat{\cF}^s_{g,x}(y)$ with $\hat{\cF}^{cu}_{g,x}(z)$, which we denote by $[y,z]$.
\end{itemize}
\end{lemma}

For $g\in \cU$ and any $x\in M$, we considering the following three types of Bowen balls:
\begin{itemize}
\item \emph{finite Bowen ball}: $B_n(g,x,\vep)=\{y\in M: d(g^i(x),g^i(y))<\vep, |i|<n\},$
\item \emph{negative Bowen ball}: $B^-_\infty(g,x,\vep)=\{y\in M: d(g^i(x),g^i(y))<\vep, i<0\},$
\item \emph{(two sided) infinite Bowen ball}: $$B_\infty(g,x,\vep)=\{y\in M: d(g^i(x),g^i(y))<\vep, i\in \mathbb{Z}\}.$$
\end{itemize}

It was shown in the proof of \cite{LVY}[Theorem 3.1] that:
\begin{lemma}\label{l.unstablefake}
For $\vep<r_0/2$ and any $x\in M$, $B^-_\infty(g,x,\vep)\subset \hat{\cF}^{cu}_{g,x}(y,2\vep)$.
\end{lemma}

We may take $r_1$ in the previous section to satisfy that
\begin{equation}\label{eq.sizeunstable}
r_1<r_0/2.
\end{equation}
Then as a consequence of Lemma~\ref{l.hyperbolictime}, we show that for every point in $\cH(b_0/2,g)$, the unstable manifold has uniform size:
\begin{lemma}\label{l.sizeunstable}
For any $x\in \cH(b_0/2,g)$, $\hat{\cF}^{cu}_{g,x}(x,r_1)\subset W_{\loc}^u(x)$. More precisely,
for any $y\in \hat{\cF}^{cu}_{g,x}(x,r_1)$,
$$d_{\hat{\cF}^{cu}_{g,g^{-n}(x)}(g^{-n}(x))}(g^{-n}(x),g^{-n}(y))\leq e^{-nb_0/2}d_{\hat{\cF}^{cu}_{g,x}(x)}(x,y).$$
\end{lemma}
The goal of this subsection is to show that the measures in $\G^u(g)$ for $g\in \cU$ are uniformly entropy expansiveness.

\begin{lemma}\label{l.uniformentropyexpansiveness}
For any $g\in \cU$, and any measure $\mu \in \G^u(g)$, for $\mu$ almost every point $x$,
$$B_\infty(g,x,r_1)=x.$$
\end{lemma}
\begin{proof}
By Lemma~\ref{l.unstablefake} and the choice of $r_1\leq r_0/2$, we have
$$B_\infty(g,x,r_1)\subset B^-_\infty(g,x,r_1)\subset \hat{\cF}^{cu}_{g,x}(x,r_1).$$

Let $x$ be a $\mu$ typical point, by Lemma~\ref{l.Pliss}, we may assume that the forward orbit of $x$ enters $\cH(b_0/2,g)$
infinitely many times. Suppose there is a distinct point $y\in B_\infty(g,x,r_1/2)\subset \hat{\cF}^{cu}_{g,x}(x,r_1)$, we are going to
prove by contradiction. Suppose $f^n(x)\in \cH(b_0/2,g)$, then $f^n(y)\in B_\infty(g,g^n(x),r_1/2)\in \hat{\cF}^{cu}_{g,x}(x,r_1)$.
By Lemma~\ref{l.sizeunstable},
$$d_{\hat{\cF}^{cu}_{g,x}(x)}(x,y)\leq e^{-nb_0/2}d_{\hat{\cF}^{cu}_{g,g^n(x)}(g^n(x))}(g^n(x),g^n(y))\leq e^{-nb_0/2}r_1.$$

Taking $n\to \infty$, we have $d_{\hat{\cF}^{cu}_{g,x}(x)}(x,y)=0$. Hence $x=y$, a contradiction with the hypothesis that $x$
and $y$ are distinct. The proof is complete.
\end{proof}

\begin{remark}
The classical definition of entropy expansive by Bowen requires that the topological entropy of $B_\infty(g,x,r_1)$ to be vanishing for {\em every} $x\in M$. However, as observed in~\cite{LVY}, this is equivalent to having zero topological entropy for the infinite Bowen ball for {\em every} invariant measure $\mu$ and {\em $\mu$ almost every} $x$.  The statement  of the previous lemma follows this approach.

Also note that this lemma does not immediate lead to the upper semi-continuity of $h_\mu$ as in the classical case, since we only have entropy expansive on a {\em subspace} of invariant measure. However, we will see in a second that the upper semi-continuity holds for measures in $\G^u$.
\end{remark}

\subsection{Upper semi-continuity of metric entropy\label{ss.uppermetricentropy}}
In this section, we are going to show that the metric entropy for measures in $\G^u(\cdot)$
is upper semi-continuous, which is a consequence of the uniform entropy expansiveness for measures among $\G^u(\cdot)$.

Define the \emph{$\vep$-tail entropy at $x$} by
$$h^*(g,x,\vep)=h_{top}(g,B_\infty(g,x,\vep)).$$
For any probability measure $\mu$ of $g$, let $h^*(g,\mu,\vep)=\int h^*(g,x,\vep) d\mu(x)$.

As a direct consequence of Lemma~\ref{l.uniformentropyexpansiveness}, we get
\begin{lemma}\label{l.entropyexpansive}
For any $g\in \cU$ and any $\mu\in \G^u(g)$, $h^*(g,\mu,r_1)=0$.
\end{lemma}

We also need the following lemma of \cite{CLY}[Theorem 1.2]:

\begin{lemma}\label{l.tail}
$h_\mu(g)-h_\mu(g,\cP)\leq h^*(g,\mu,\rho)$ for any finite measurable partition $\cP$ with $diam(\cP)\leq \rho$.
\end{lemma}
By Lemma~\ref{l.entropyexpansive} and Lemma~\ref{l.tail},
we conclude that $h_\mu(g)=h_\mu(g,\cP)$ for any finite measurable partition $\cP$ with $diam(\cP)\leq \rho$. In particular,
by a standard argument for upper semi-continuity of metric entropy (see for instance \cite{LVY}[Lemma 2.3]), we have:
\begin{corollary}\label{c.uppersemicontinuous}
Let $g_n$ ($n\geq 0$) be a sequence of $C^1$ partially hyperbolic diffeomorphisms inside $\cU$, and
$\mu_n\in \G^u(g_n)$. Suppose $g_n\to g_0$ in $C^1$ topology and $\mu_n \to \mu_0\in \G^u(g_0)$
in weak-* topology, then
$$\limsup_{n\to \infty}h_{\mu_n}(g_n)\leq h_{\mu_0}(g_0).$$
\end{corollary}

\subsection{Physical measures\label{ss.physicalmeasures}}
In this section, we will provide a uniform treatment on the existence of physical measures for
all $C^{1+}$ diffeomorphisms in $\cU$. For this purpose, let $r_1>0$ be given by Lemma~\ref{l.hyperbolictime} and~\eqref{eq.sizeunstable}.
\begin{proposition}\label{p.existencephysical}
Let $g$ be any $C^1$ diffeomorphism of $\cU$. Then $\G(g)$ is compact and convex, and every extreme element of
$\G(g)$ is an ergodic measure. The map:
$\cG:$ $g\in \cU \mapsto \G(g)$ is upper semi-continuous with respect to diffeomorphisms in $\cU$ under $C^1$ topology.
Moreover, if $g$ is $C^{1+}$, then $\G(g)$ has finitely many extreme points, each of which is a physical measure of $g$ and vice versa. The basin of each physical measure of
$g$ contains Lebesgue almost every point of some ball with radius $r_1$.
\end{proposition}
\begin{proof}
Recall that
$$\G^{cu}(g)=\{\mu\in \cM_{\inv}(f): h_\mu(g)\geq \int \log(\det(Dg\mid_{E^{cu}(x)}))d\mu(x)\}.$$
Because the metric entropy function is affine, it follows that $\G^{cu}(f)$ is convex.
By Proposition~\ref{p.Gu}, $\G^u(g)$ is convex, so is $\G(g)=\G^u(g)\bigcap \G^{cu}(g)$.

The compactness of $\G(g)$ follows from Corollary~\ref{c.uppersemicontinuous}. More precisely, suppose
there is a sequence of invariant probabilities $\{\mu_n\}_{n=0}^\infty$ of $g$ such that $\mu_n\in \G(g)$
and assume $\lim_{n\to \infty} \mu_n=\mu$. Because $\mu_n\in \G^{cu}(g)$, we have
$$h_{\mu_n}(g)\geq \int \log(\det(Dg\mid_{E^{cu}(x)}))d\mu_n(x).$$
Note that $\mu_n\in \G^{u}(g)$, and by Proposition~\ref{p.Gu}, $\G^u(g)$ is compact, we have
$\mu\in \G^u(g)$. It then follows from Corollary~\ref{c.uppersemicontinuous} that $\limsup_{n\to \infty} h_{\mu_n}(g)\leq h_\mu(g)$, which implies:
$$h_{\mu}(g)\geq \int \log(\det(Dg\mid_{E^{cu}(x)}))d\mu(x).$$
This means $\mu\in \G^{cu}(g)$, thus $\mu\in \G^u(g)\cap \G^{cu}(g)=\G(g)$.

Indeed, by Corollary~\ref{c.uppersemicontinuous} and a similar proof as above, for a sequence of $C^1$ maps $g_n\in \cU,g_n\to g\in\cU$ in $C^1$ topology and $\mu_n\in \G(g_n)$ converging to $\mu$ in weak-* topology, we have $\mu\in \G(g).$ Then the map $\cG(\cdot)$ is upper semi-continuous, as claimed.

Suppose that $\mu$ is any extreme element of $\G(g)$, then
it is contained in $\G^u(g)$. We claim that:

\begin{lemma}\label{l.ergodicdecompo}
$\mu$ is ergodic.
\end{lemma}
\begin{proof}
Let $\tilde{\mu}$ be a typical ergodic component in the ergodic decomposition of $\mu$, we are going to show that $\tilde{\mu}\in \G(g)$;
this implies that $\tilde{\mu}$ is also an extreme element of $\G(g)$, thus it coincides with $\mu$.

By Proposition~\ref{p.Gu}, $\tilde{\mu}\in \G^u(g)$.  Thus it suffices to show that $\tilde{\mu}\in \G^{cu}(g)$.

Because $g\in \cU$, by Lemma~\ref{l.robustGexpand}, any measure $\nu\in \G^u(g)$ has positive center exponent. By Ruelle's inequality,
$$h_{\tilde{\mu}}(g)\leq  \int \log(\det(Dg\mid_{E^{cu}(x)}))d\tilde{\mu}(x).$$

Because $\mu\in \G^{cu}(g)$,
$$h_\mu(g)\geq  \int \log(\det(Dg\mid_{E^{cu}(x)}))d\mu(x).$$
Since entropy function is an affine functional with respect to invariant measures, we must have
$h_{\tilde{\mu}}(g)= \int \log(\det(Dg\mid_{E^{cu}(x)}))d\tilde{\mu}(x)$ for typical ergodic component $\tilde{\mu}$ of $\mu$.
Thus $\tilde{\mu}\in \G^{cu}(g)$.
The proof is complete.
\end{proof}

We continue the proof of Proposition~\ref{p.existencephysical}. Assume that $g\in \cU$ is a $C^{1+}$ partially hyperbolic diffeomorphism.
First we suppose that $\mu$ is an extreme element of $\G(g)$. Then by the discussion above,
$\mu$ is  ergodic with positive center exponents. Moreover, by Ruelle's inequality, we get
$$h_\mu(g)= \int \log(\det(Dg\mid_{E^{cu}(x)}))d\mu(x).$$

By the entropy formula of Ledrappier-Young \cite{LY85a}, the disintegration of $\mu$ along the Pesin unstable manifold is equivalent to
the Lebesgue measure on the leaves. This means, for $\mu$ almost every $x$, Lebesgue almost every point on the Pesin unstable manifold of $x$ is a typical point of $\mu$. Since the Basin of $\mu$ is saturated by stable leaves (we use the fact that $E^s$ is uniformly contracting), and the stable
foliation is absolutely continuous, the union of the stable leaves of the previous full Lebesgue measure subset
of $W^u(x)$ is contained in the basin of $\mu$ and has full volume inside a ball with center at $x$. Note however, that such ball may not have uniform radius $r_1$.

To obtain a ball with radius $r_1$ in the basin of $\mu$, we apply Lemma~\ref{l.Pliss} to obtain an $n>0$, such that $g^n(x)\in \cH(b_0/2,g)$.
Then by Lemma~\ref{l.sizeunstable}, $W^u(g^n(x),g)$ contains a disk with radius $r_1$, where Lebesgue typical points in this disk are typical points of $\mu$. By the uniform transversality between
the bundles $E^s$ and $E^{cu}$, the basin
of $\mu$ contains Lebesgue almost every point of a ball  at $g^n(x)$ with radius $r_1$, which we denote by $B_{g^n(x)}(r_1)$. It then follows that
\begin{equation}\label{eq.support}
\mu(B_{g^n(x)}(r_1))>0.
\end{equation}
To simplify notation, we write any ball obtained in the above way by $B_{\mu}$.

Because the basins
of different physical measures are disjoint, $\G(g)$ has only finitely many extreme elements. We denote them by $\mu_1,\cdots, \mu_k$.

Now we prove that the union of basins of $\mu_1,\cdots,\mu_k$ has full volume. We prove by contradiction, suppose the compliment
of $\bigcup_{i=1}^k \cB(\mu_i)$, we denote by $\Lambda$, has positive volume.
By Proposition~\ref{p.physical}, for Lebesgue almost every point
$x\in \Lambda$, any limit $\mu$ of the sequence $\frac{1}{n}\sum_{i=0}^{n-1}\delta_{g^i(x)}$ belongs to $\G(g)$.
We may choose the previous $x$ to be a Lebesgue density point of $\Lambda$, and denote by
$\mu=\lim_i\frac{1}{n_i}\sum_{j=0}^{n_i-1}\delta_{g^j(x)}$.

Because $\G(g)$ is convex with finitely many extreme elements, $\mu$ can be written as a combination:
$$\mu=a_1 \mu_1+\cdots +a_k \mu_k$$
where $0\leq a_1,\cdots, a_k \leq 1$ and $\sum_{i=1}^k a_k=1$. There is
$1\leq t\leq k$ such that $a_t>0$. Then by \eqref{eq.support}, $\mu(B_{\mu_t})\geq a_t \mu_t(B_{\mu_t})>0$.

Thus there is $n_i$ sufficiently large,
$\frac{1}{n_i}\sum_{j=0}^{n_i-1} \delta_{g^j(x)} (B_{\mu_t})>0$. In particular, there is $j>0$ such that $g^{j}(x)\in B_{\mu_t}$.

Because we choose $x$ a Lebesgue density point of $\Lambda$, i.e., it satisfies:
$$\lim_{r\to 0^+} \frac{\Leb(B_x(r)\cap \Lambda)}{\Leb(B_x(r))}\to 1.$$ Observe that since the basin of physical measures is invariant under
$g$, $\Lambda$ is invariant under the iteration of $g$ also. Then the above argument shows that $\Lambda \bigcap B_{\mu_j}=g^{j}(\Lambda)\bigcap B_{\mu_j}$ has positive Lebesgue measure.

Recall that Lebesgue almost every point of $B_{\mu_j}$ is in the basin of $\mu_j$. Therefore $\Lambda$ and the basin of $\mu_j$ have non-trivial intersection.
This contradicts the choice of $\Lambda$. The proof of Proposition~\ref{p.existencephysical} is complete.
\end{proof}

\begin{remark}
The $C^{1+}$ regularity is used to:
\begin{itemize}
\item show that the conditional measures of $\mu$ along unstable leaves are absolutely continuous; we need the work of Ledrappier and Young, which requires $C^{1+}$;
\item show that the basin of $\mu$ contains Lebesgue almost every point in a ball; there we need the stable foliation to be absolutely continuous.
\end{itemize}
We will see later in Section~\ref{s.robustskeleton} that such regularity condition can be bypassed for generic $C^1$ diffeomorphisms in $\cU$.
\end{remark}

\section{Proof of Theorem~\ref{main.sksleton} and Corollary~\ref{maincor.period}}\label{s.6}
In this section, we provide the proof of Theorem~\ref{main.sksleton} and Corollary~\ref{maincor.period}.

Throughout this section, we assume $f$ to be a $C^{1+}$ diffeomorphism with mostly expanding center,
$\cU$ a sufficiently small $C^1$ neighborhood of $f$. By Proposition~\ref{p.onestep}, there is $b_0>0$ such that for any
$C^1$ diffeomorphism $g\in \cU$ and any $\mu\in\G^u(g)$,
\begin{equation}\label{eq.Guexpanding}
\int \log \|Dg^{-1}\mid_{E^{cu}(x)}\|d\mu(x)<-b_0.
\end{equation}
The structure of this section is as following:
In Section~\ref{ss.Liao} we introduce the Liao's shadowing lemma, which will be used in  Section~\ref{ss.skeleton} to construct skeletons. For the discussion in Section~\ref{s.8}, we will make the construction for every $C^1$ diffeomorphism $g\in\cU$.

Then in Section~\ref{ss.skeletonandmeasures}, we will show that each element in $\cS(g)$ is associated to a physical measure, assuming that $g$ is $C^{1+}$. This concludes the proof of Theorem~\ref{main.sksleton}.
Finally, in Section~\ref{ss.proofAB} we provide the proof of Corollary~\ref{maincor.period}.

\subsection{Liao's shadowing lemma\label{ss.Liao}}
\begin{definition}\label{df.pseudo}
An orbit segment $(x,f(x),\cdots,f^n(x))$ is called \emph{$\lambda$-quasi hyperbolic} if there exists
$0< \lambda <1$ such that
\begin{equation}
\prod_{i=0}^{k-1} \|Df^{-1}\mid_{E^{cu}(f^{n-i}(x))}\|< \lambda^k
\end{equation}
for any $1 \leq k \leq n$.
\end{definition}
In other words, $(x,f(x),\cdots,f^n(x))$ is $\lambda$-quasi hyperbolic if $n$ is a $(-\log \lambda)$-hyperbolic time for $x$.
In this subsection we need the following shadowing lemma by Liao, which allows a quasi hyperbolic, periodic pseudo orbit to be shadowed by a periodic orbit with large unstable manifold.
\begin{lemma}[\cite{Liao,Gan}]\label{l.Liao}
For any $\lambda>0$, there exist $\rho>0$ and $L>0$, such that for any $\lambda$-quasi hyperbolic orbit
$(x,f(x),\cdots,f^n(x))$ of $f$ with $d(x,f^n(x))\leq \rho$, there
exists a hyperbolic periodic point $p\in M$ such that
\begin{itemize}
\item[(a)] $p$ is a hyperbolic periodic point with period $n$ and with stable index $i_s$;
\item[(b)] $d(f^i(x),f^i(p))\leq L d(x,f^n(x))$ for any $0\leq i \leq n-1$;
\item[(c)] $p$ has uniform size unstable manifold: there is a constant $r>0$ depending on $\lambda$, such that the local unstable manifold of $p$ contains a disk with radius $r$.
\end{itemize}
\end{lemma}

\begin{remark}\label{rk.uniformshadowing}
The parameters in the previous lemma can be made uniform for diffeomorphisms in a $C^1$ neighborhood $\cU$ of $f$.
Moreover, one can take $\delta$ sufficiently small, then $d(f^i(x),f^i(p))\leq L d(x,f^n(x))$ is sufficiently small
for any $0\leq i \leq n-1$, and then
$$
\prod_{i=0}^{k-1} \|Df^{-1}\mid_{E^{cu}(f^{n-i}(p))}\|\leq \lambda^k
$$
for any $1 \leq k \leq n$. In particular, if one takes $\lambda = e^{-b_0/2}$, then the size of unstable manifold of $p$ can be chosen to be $r_1>0$, which is the constant given by Lemma~\ref{l.sizeunstable}.
\end{remark}
\begin{definition}\label{df.lambdapt}
A periodic point $p$ of $g\in \cU$ is called a \emph{$\lambda$-hyperbolic periodic point}, if it satisfies
\begin{equation}\label{eq.lambdapt}
\prod_{i=0}^{k-1} \|Df^{-1}\mid_{E^{cu}(f^{n-i}(p))}\|\leq \lambda^k
\end{equation}
for any $1 \leq k \leq \pi(p)$.
\end{definition}

By the previous discussion and Remark~\ref{r.uniformunstable}, we have shown that
\begin{lemma}\label{l.uniformsizeunstable}
For any $e^{-b_0/2}$-quasi hyperbolic periodic point, its unstable manifold contains a $r_1$-ball
inside the $cu$-fake leaf $\hat{\cF}^{cu}_{g,p}(p,r_1)$
\end{lemma}
\subsection{Existence of skeleton\label{ss.skeleton}}
In this section, we will show that any $C^1$ diffeomorphism $g\in \cU$ admits a skeleton.
The main result of this section is Proposition~\ref{l.existenceskeleton}.

In order to apply Liao's shadowing lemma, we need to establish the existence of orbit segments that are quasi-hyperbolic. This follows from Proposition~\ref{p.physical} and \eqref{eq.Guexpanding}:
\begin{proposition}\label{p.aeptexpanding}
Suppose $g\in \cU$. There is a full volume subset $\Gamma_g$ such that for Lebesgue almost every point $x\in \Gamma_g$,
$$\limsup_{n\to \infty}\frac{1}{n}\sum_{i=0}^{n-1} \log\|Dg^{-1}\mid_{E^{cu}(g^{n-i}(x))}\|\leq -b_0.$$
\end{proposition}

By the Pliss Lemma (see \cite{ABV00}), there exists $n_x\geq 1$ and $\delta_1>0$ such that
\begin{equation}\label{eq.expandingorbit}
\#(H(b_03/4,x,g)\cap [1,n))\geq n\delta_1 \text{ for all $n\geq n_x$ },
\end{equation}
where $H(b_03/4,x,g)$ is the collection of $b_03/4$-hyperbolic times along the forward orbit of $x$.

Taking a sequence of integers $n_x\leq n_1< n_2 <\cdots$ such that $n_i\in H(b_03/4,x,g))$, we may assume that
$x_{n_i} = f^{n_i}(x)$ converges to a point $x_0$. For $\lambda=e^{-\frac{3b_0}{4}}$, $\rho$ and $L$ are obtained by Lemma~\ref{l.Liao}.
We may further assume that $\sup_{i,j}\{d(x_{n_i},x_{n_j})\}\leq  \rho_0 \leq \rho$ where $\rho_0$ satisfies that
for any two points $y,z\in M$ with $d(y,z)\leq L \rho_0$, we have
\begin{equation}\label{eq.close}
\mid \log\|Dg^{-1}\mid_{E^{cu}(y)}\|-\log\|Dg^{-1}\mid_{E^{cu}(z)}\| \mid \leq b_0/4.
\end{equation}

Because for any $i<j$, the pseudo orbit $\{x_{n_i},x_{n_i+1},\cdots,x_{n_j-1}\}$ is $b_03/4$-quasi hyperbolic,
by Lemma~\ref{l.Liao}, this pseudo orbit is $L d(x_{n_i},x_{n_j})\leq L \rho_0$ shadowed by a periodic orbit $p_{x,i,j}$.
Because $x_{n_i}\to x_0$ as $i\to\infty$, all the periodic points $p_{x,i,i+1}$ converge to $x_0$.

Moreover, by the choice of $\rho_0$ in \eqref{eq.close}, $p_{x,i,j}$ is a $e^{-b_0/2}$-quasi hyperbolic periodic point.
By Lemma~\ref{l.uniformsizeunstable}, each periodic point $p_{x,i,j}$ has unstable manifold with size at least $r_1$. Their stable manifold already have uniform size due to $E^s$ being uniformly contracting (note that all $p_{x,i,j}$'s have stable index $i_s$). Thus there
is $m_x$ such that for any $i,j>m_x$, $p_{x,i,i+1}$ and $p_{x,j,j+1}$ are homoclinic related to each other,
and
\begin{equation*}\label{e.2}
\cF^s_{\loc}(x_{i})\pitchfork W^u_{r_1}(p_{j,j+1})\neq \emptyset.
\end{equation*}
Furthermore, $\cF^s_{\loc}(p_i)$ will intersect transversally with any disk center at $x_j$, tangent to the $cu$ cone with radius at least $r_1$.

To simplify notation, we will write  $p_{x,i,i+1}=p_{x,i}$.

\begin{lemma}\label{l.xskeleton}
For any $i>m_x$, $x\in \Cl(\cF^s(\Orb(p_{x,i})))$.
\end{lemma}
\begin{proof}
Let $U$ be any small neighborhood of $x$. Because for any $i>m_x$, all the hyperbolic periodic points $p_{i}$ are
homoclinic related to each other, we only need to show that there is $i>m_x$, such that $\cF^s(\Orb(p_i))\cap U\neq \emptyset$, then the lemma will follow from the Inclination lemma.

We take $\vep>0$ small enough such that $\hat{\cF}^{cu}_{g,x}(x,\vep)\subset U$. By Lemma~\ref{l.hyperbolictime}, for
$i>m_x$ sufficiently large, $g^{i}(\hat{\cF^{cu}}_{g,x}(x,\vep))\supset \hat{\cF^{cu}}_{g,x_{n_i}}(x_{n_i},r_1)$, where the latter is a disk
tangent to a $cu$ cone with uniform diameter. This means that when $i$ is sufficiently large,
$$g^{n_i}(\hat{\cF^u}_{g,x}(x,\vep))\pitchfork \cF^s_{\loc}(p_{i})\neq \emptyset.$$
By the invariance of the stable manifold under the iteration of $g$, we have $U\cap \cF^s(\Orb(p_i))\neq \emptyset$.

The proof is complete.
\end{proof}

Now we are ready to construct the skeleton for $g\in\cU$.
By Proposition~\ref{p.aeptexpanding}, for each $x\in \Gamma_g$,
we fix any one of $p_{x,i}$ for $i>m_x$ and denote it by $p_x$. Then by the previous lemma,
the union $\bigcup_{x\in \Gamma}\cF^s(\Orb(p_x))$ is dense in the manifold $M$.

Moreover, since each periodic point $p_x$ has stable and unstable manifold with size at least $r_1$, there are
only finitely many of them that are not homoclinically related to each other, with number uniformly bounded from above. Take $\{p_1,\cdots,p_k\}$ a subset
of $\{p_x\}_{x\in \Gamma}$ which are not homoclinic related and has maximal cardinality.

We claim that
$\bigcup_{i=1,\cdots,k}\cF^s(\Orb(p_i))$ is dense in the manifold $M$. Assume that this is not the case, then we can take $p_x$ for $x\in M\setminus \bigcup_{i=1,\cdots,k}\Cl(\cF^s(\Orb(p_i)))$. By the choice of $\{p_1,\cdots,p_k\}$, $p_x$ must be homoclinically related to some  $p_i$. However, this means that $\Cl(\cF^s(\Orb(p_i)))=\Cl(\cF^s(\Orb(p_x)))$ by the Inclination lemma. Lemma~\ref{l.xskeleton} then shows that $x\in\Cl(\cF^s(\Orb(p_i))) $, which is a contradiction.

Thus
$\{p_1,\cdots,p_k\}$ forms a pre-skeleton. By Lemma~\ref{l.subsetofpre}, we have shown that:

\begin{proposition}\label{l.existenceskeleton}
Every $C^1$ diffeomorphism $g\in\cU$ admits a skeleton $\cS(g)=\{p_1,\cdots,p_k\}$,
such that for any $1\leq i \leq k$, $W^u(p_i)$ contains a ball in the fake $cu$ leaf with center at $p_i$ and radius $r_1$.
\end{proposition}

From now on, we fix $\cS(g)=\{p_1,\cdots,p_k\}$ a skeleton obtained as above.



\subsection{Skeleton and measures\label{ss.skeletonandmeasures}}
In this section we assume $g\in \cU$ to be $C^{1+}$, then by Lemma~\ref{l.robustGexpand}, $g$ has mostly expanding center. We will establish a one-to-one correspondence between elements of $\cS(g)$ and the physical measures of $g$.

By Proposition~\ref{p.existencephysical}, $g$ has only finitely many physical measures $\{\mu_1,\cdots,\mu_l\}$. Moreover, from
Lemma~\ref{l.Pliss} and Proposition~\ref{p.existencephysical}, there is $r_1>0$ only depending on $\cU$ and $b_0$ such that,
for any physical measure $\mu_j$ of $g$, there is a $\mu_j$ regular point $x_j$, such that:
\begin{itemize}
\item[(a)] $x_j\in \cH(b_0/2,g)$, thus has Pesin unstable manifold with size larger than $r_1$;
\item[(b)] $\mu$ regular points consists of Lebesgue almost every point on the Pesin unstable manifold of $x_j$.
\end{itemize}
The main result of this section is the following:

\begin{proposition}\label{l.support}
The number of physical measures and the number of elements of skeleton of $g$ are the same, i.e., $k=l$.
Indeed, there is a bijective map: $j\to i(j)$ such that for any physical measure $\mu_j$ of $g$,
there is $p_{i(j)}\in \cS(g)$ such that $\supp(\mu_j)= \Cl(W^u(\Orb(p_i),g))$, and Lebesgue almost every
point on $W^u(\Orb(p_i),g)$ belongs to the basin of $\mu_j$. Moreover, the closure of $\cF^s(\Orb(p_i))$ coincides with the
closure of $\cB(\mu_j)$.
\end{proposition}
\begin{proof}
Fix any physical measure $\mu_j$ of $g$. By (a) above, there is $p_i\in \cS(g)$ such that
$\cF^s(\Orb(p_i))\pitchfork W^u_{r_1}(x_j,g)\neq \emptyset$. By the Inclination lemma, $g^n(W^u_{r_1}(x_j,g))$ converges
to $W^u(\Orb(p_i),g)$. Because $W^u_{r_1}(x,g)\subset \supp(\mu_j)$ by (b) above, we have $\Cl(W^u(\Orb(p_i),g))\subset \supp(\mu_j)$.

To show the reversed inclusion, note that for $n$ large enough, by the Inclination lemma, $g^n(W^u_{r_1}(x,g))$ approaches $W^u_{\loc}(p)$ in the following sense:
there is stable holonomy map from $W^u_{\loc}(p)$ to $g^n(W^u_{r_1}(x,g))$ induced by the stable foliation.
Because the set of $\mu_j$ typical points is invariant under iteration, Lebesgue almost every point of $g^n(W^u_{r_1}(x,g))$
is also typical for $\mu_j$. Since stable foliation is absolutely continuous, and the basin of $\mu_j$ is $s$-saturated,
it follows that  Lebesgue
almost every point of $W^u(\Orb(p_i),g)$ belongs to the basin of $\mu_j$.

Take any point $y\in W^u(p_i)\cap \cB(\mu_j)$. Because $g^n(y)\in W^u(\Orb(p_i))$ for any $n\geq 0$, thus $\mu_j=\lim \frac{1}{n}\sum \delta_{g^i(y)}$ is supported on $\Cl(W^u(\Orb(p_i),g))$.
As a conclusion,
\begin{equation}\label{eq.unstablesupport}
\supp(\mu_j)=\Cl(W^u(\Orb(p_i),g)).
\end{equation}

Because Lebesgue almost every point on $W^u(\Orb(p_i),g)$ belongs to the basin of $\mu_j$, the map
$j\to i(j)$ is injective; in particular, we have $k\geq l$. After reordering the periodic points of $\cS(g)$, we may assume that $i(j)=j$ for $j=1,\cdots, l$.

In order to prove $k=l$, we only need to show that $\{p_1,\cdots,p_l\}$ is a pre-skeleton, i.e.,
$\bigcup_{i=1}^l \cF^s(\Orb(p_i))$ is dense in the manifold $M$.
By Proposition~\ref{p.existencephysical}, the union of basins of physical measures has full volume, thus it suffices
to prove that  for each $1\leq i \leq l$, the closure of $\cF^s(\Orb(p_i))$ coincides with the
closure of $\cB(\mu_i)$.

By \eqref{eq.unstablesupport} we have $p_i\in \supp(\mu_i)$. Take $r>0$ sufficiently small such that $\mu_i(B_r(p_i))>0$ and
$B_r(p_i)\subset \cO_i=\bigcup_{y\in W^u(\Orb(p_i),g)}\cF^s(y)$. For any $x\in \cB(\mu_i)$, since we have
$\frac{1}{n}\sum_{i=0}^{n-1}\delta_{f^i(x)}\to \mu_i$, there is $n$ sufficiently large such that $\frac{1}{n}\sum_{i=0}^{n-1}\delta_{f^i(x)}(B_r(p_i))>0$. This shows that there is $m>0$
such that $f^m(x)\in B_r(p_i)\subset \cO_i$. By (ii) of Proposition~\ref{p.skeletongeometry}, $f^m(x)$ is accumulated by $\cF^s(\Orb(p_i))$, so is $x$. Thus we have shown that the basin of $\mu_i$ is contained in the closure of $\cF^s(\Orb(p_i))$, while the reversed inclusion follows immediately from the $u$-saturation of $\supp(\mu_i)$. We now conclude that $k=l$.

The proof is complete.
\end{proof}
\begin{proof}[Proof of Theorem~\ref{main.sksleton}]
By Proposition~\ref{l.existenceskeleton}, $f$ admits an index $i_s$ skeleton. Let
$\cS=\{p_1,\cdots,p_k\}$ be any index $i_s$ skeleton of $f$. By Proposition~\ref{l.support},
the number of physical measures is precisely $k=\#\cS$,
and for each $p_i\in \cS$ there exists a distinct physical measure $\mu_i$ such that
\begin{itemize}
\item[(1)] the closure of $W^u(Orb(p_i))$ coincides with $\supp(\mu_i)$ and by (ii) of Proposition~\ref{p.skeletongeometry}, they also coincide with the homoclinic class of the orbit $\Orb(p_i)$.
\item[(2)] the closure of $\cF^s(\Orb(p_i))$ coincides with the closure of the basin
of the measure $\mu_i$.
\end{itemize}
Moreover, by (ii) of Proposition~\ref{p.skeletongeometry},
$$\Int(\Cl(\cB(\mu_i)))\cap \Int(\Cl(\cB(\mu_j)))= \emptyset$$
for $1 \leq i \neq j \leq k$.
The proof is finished.
\end{proof}
\subsection{Proof of Corollary~\ref{maincor.period}\label{ss.proofAB}}
We finish this section with the proof of Corollary~\ref{maincor.period}.

\begin{proof}[Proof of Corollary~\ref{maincor.period}]
Let $f$ be $C^{1+}$. For any $n>0$, and $\nu$ an ergodic Gibbs $u$-state of $f^n$, by Lemma~\ref{l.power},
$\mu_n=\frac{1}{n}\sum_{i=0}^{n-1}f^i_*(\nu)$ is an invariant Gibbs $u$-state of $f$. It is easy to see that for
$\nu$ typical point $x$, its center exponents with  respect to $f^n$ are $n$ times of the corresponding exponents respect to $f$.
In particular, the center exponents of every Gibbs $u$-state of $f^n$ are positive. Thus
$f^n$ has mostly expanding center as well.

Because $\{p_1,\cdots,p_k\}$ is an index $i_s$ skeleton of $f$, $\bigcup_i \bigcup_{q\in \Orb(p_i)} \cF^s(q)$
is dense in the manifold $M$, which means that $\cS = \{q\in\Orb(p_i),i=1,\ldots,k\}$ is a pre-skeleton of $f^n$ for every $n\ge 0$. By Lemma~\ref{l.subsetofpre}, it has a subset which is a skeleton of $f^n$.
It follows from Theorem~\ref{main.sksleton} that $f^n$ has finitely many physical measures
with number bounded by $P=\prod_{i=1}^k \pi(p_i) = \# S$.

Moreover, because elements of $\cS$ are all distinct fixed points of $f^{nP}$ for any $n>0$, it is a skeleton for $f^{nP}$, $n>0$. Then by Theorem~\ref{main.sksleton}, $f^{nP}$ have the same number of
physical measures for every $n>0$. Let $\mu$ be a physical measure of $f^P$. By Proposition~\ref{p.existencephysical}, $\mu$ is ergodic,
and its conditional measures along the Pesin unstable manifolds are equivalent to the Lebesgue measure on the leaves. Below we will show that $\mu$ is ergodic for $f^{nP}$ for all $n>0$.

To this end, let $\tilde{\mu}$ be any ergodic component of $\mu$ with respect to $f^{nP}$, then the conditional measures  of $\tilde{\mu}$ along the Pesin
unstable manifolds are still equivalent to the Lebesgue measure on the leaves. It then follows from the argument of
Proposition~\ref{p.existencephysical} that $\tilde{\mu}$ is a physical measure of $f^{nP}$. Since the number of physical measures of
$f^{nP}$ are constant, $\tilde{\mu}$ must be the only ergodic component of $\mu$ with respect to $f^{nP}$. It then follows that $\mu = \tilde{\mu}$ which is ergodic for $f^{nP}$.

Then, by the classical work of Ornstein and Weiss \cite{OW}, every physical measure of $f^P$ is a Bernoulli measure.
\end{proof}

\section{Proof of Theorem~\ref{main.robustskeleton}\label{s.robustskeleton}}
In this section, we study the robustness of the skeleton and physical measures under $C^1$ topology among $C^{1+}$ diffeomorphisms and prove Theorem~\ref{main.robustskeleton}

For this purpose, we assume that $f:M\to M$ is a $C^{1+}$ partially hyperbolic diffeomorphism with mostly expanding center,
and $\cU$ a $C^1$ neighborhood of $f$ satisfying Lemma~\ref{l.robustGexpand} and Proposition~\ref{p.existencephysical}.
Let $b_0$ be given in Lemma~\ref{l.robustGexpand} and $r_1$ be given by Proposition~\ref{p.existencephysical}.
We take $\cS(f)=\{p_1,\cdots, p_k\}$ a skeleton of $f$. Since $\bigcup_{i=1}^k\cF^s(\Orb(p_i(f)),f)$ is dense in the manifold $M$, by the continuity of stable
foliation with respect to diffeomorphisms in $C^1$ topology, we may assume that $\cU$ is sufficiently small such that for any $C^1$ diffeomorphism $g\in \cU$, the continuation of $\cS(f)$ given by the continuation of hyperbolic saddles:
$\cS(g) = \{p_i(g),\cdots, p_k(g)\}$ satisfies that $\bigcup_{i=1}^k\cF^s(\Orb(p_i(g)),g)$ is $r_1$ dense, i.e., for any $x\in \cH(b_0/2,g)$,
$$\bigcup_{i=1}^k\cF^s(\Orb(p_i(g)),g)\pitchfork W^u_{r_1}(x,g)\neq \emptyset,$$
where $W^u_{r_1}(x,g)$ is given by Lemma~\ref{l.sizeunstable}.

Note that $\cS(g)$ may not be a skeleton. In the following, we will show the relation
between skeletons of diffeomorphisms in $\cU$. For the discussion in the next section, we will state the following lemma for $C^1$ diffeomorphisms in $\cU$.

\begin{lemma}\label{l.robustskeleton}
For $C^1$ diffeomorphisms in $\cU$, the number of elements of skeleton varies upper semi-continuously. More precisely, for $g\in \cU$:
\begin{itemize}
\item[(1)] $\cS(g) = \{p_1(g),\cdots, p_k(g)\}$ is a pre-skeleton of $g$, thus it contains a subset which is a skeleton of $g$;
\item[(2)] suppose that $\{q_1(g),\cdots, q_l(g)\}$ is a skeleton of $g$, then there is a $C^1$ neighborhood $\cV$ of $g$ such that for
any $h\in \cV$, $\{q_1(h),\cdots, q_l(h)\}$ is a pre-skeleton of $h$.
\end{itemize}
\end{lemma}
\begin{proof}
By Proposition~\ref{l.existenceskeleton}, $g$ admits a skeleton $\{q_1(g),\cdots,q_l(g)\}$
and each $q_j(g)$ ($j=1,\cdots,l$) has unstable manifold with size $r_1$. Then by the previous assumption on $\cU$, for
every $1\leq j \leq l$, there is an $1\leq i \leq k$ such that $\cF^s(\Orb(p_i(g)),g)\pitchfork W^u_{r_1}(q_j(g),g)\neq \emptyset$. Thus, by the Inclination lemma,
$\cF^s(\Orb(q_j(g)),g)$ is accumulated by $\cF^s(\Orb(p_i(g)),g)$, which implies that
$\cup_{i=1}^k \cF^s(\Orb(p_i(g)))$ is dense in the manifold $M$. This finishes the proof of (1).

The proof of (2) is quite similar. Take $\cV$ sufficiently small such that for any $C^1$ diffeomorphism $h\in \cV$,
the continuation $\{q_1(h),\cdots, q_l(h)\}$ satisfies the condition that $\cup_{i=1}^k\cF^s(\Orb(q_i(h)),h)$ is $r_1$ dense.
By Proposition~\ref{l.existenceskeleton}, every $h\in \cV\subset \cU$ admits a skeleton $\{q^\prime_1(h),\cdots,q^\prime_t(h)\}$.
has unstable manifold with size $r_1$. Then for every $1\leq j \leq t$, there is an $1\leq i \leq l$ such that $\cF^s(\Orb(q_i(h)),h)\pitchfork W^u_{r_1}(q^\prime_j(h),h)\neq \emptyset$. Thus, by the Inclination lemma, $\cF^s(\Orb(q^\prime_j(h)),h)$ is accumulated by $\cF^s(\Orb(q_i(h)),h)$, which implies that $\cup_{i=1}^k \cF^s(\Orb(q_i(h)),h)$ is dense in $M$. This finishes the proof of (2).
\end{proof}
Thus, by Lemma~\ref{l.skeletondifferent}, the number of elements of the skeleton of $g$ is bounded from above by $k=\#\cS(f)$.
It follows that, restricted to an $C^1$ open
and dense subset $\cU^\circ\subset \cU$, the number of elements of a skeleton for diffeomorphisms of $\cU^\circ$
is locally constant. More precisely, for any $1\leq i \leq k$, denote by
$$\cU_i=\{g\in \cU; \text{skeleton of $g$ has less than $i$ number of elements.}\}$$
Then $\cU_i$ is an open set, and $\cU^\circ =\cU_1 \bigcup_{2\leq i\leq k} (\cU_i \setminus \Cl(\cU_{i-1}))$ satisfies
our requirement.

By Theorem~\ref{main.sksleton}, the number of physical measures for $C^{1+}$ diffeomorphisms in $\cU^\circ$
is locally constant.

Suppose $f_n\in \cU^\circ$ be a sequence of $C^{1+}$ diffeomorphisms such that $f_n\to f_0\in \cU^\circ$.
We assume that all $f_n$ have $m\le k$ physical measures. By the previous argument, all the diffeomorphisms $f_n$ and $f_0$
have the same number of elements in their skeletons. In particular, by Lemma~\ref{l.robustskeleton}, we may take a skeleton $\cS(f_0)=\{p_1(f_0),\cdots,p_m(f_0)\}$ of $f_0$ such that its continuation
$\cS(f_n)=\{p_1(f_n),\cdots,p_m(f_n)\}$ is a skeleton of $f_n$. For $f_n$ ($n\geq 0$), denote by $\mu_{n,1},\cdots,\mu_{n,m}$ the physical
measures of $f_n$ associated with the periodic point $p_j(f_n)$ as explained in Theorem~\ref{main.sksleton}. In the following
we are going to show that:
\begin{lemma}\label{l.weak*}
$\mu_{n,i}\overset{weak *}{\longrightarrow} \mu_{0,i}$.
\end{lemma}
\begin{proof}
For simplicity, we will only prove it for $i=1$. We prove by contradiction, and
assume (after taking subsequence if necessary) that $\mu_{n,1}\overset{weak *}{\longrightarrow} \mu \neq \mu_{0,1}$.

By Proposition~\ref{p.existencephysical},
the space $\G(\cdot)$ is compact and convex; extreme elements of $\G(\cdot)$ are precisely those  physical measures, and
it varies in a upper semi-continuous fashion with respect to diffeomorphisms in $\cU$ under $C^1$ topology.
Thus $\mu_{n,1}\in \G(f_n)$ and $\mu\in \G(f_0)$. Moreover, $\mu$ can be written as a combination
of the physical measures of $f_0$:
$$\mu=a_1\mu_{0,1}+\cdots a_m \mu_{0,m}.$$
By our assumption, $a_1\neq 1$, thus there is $1< i \leq m$ such that $a_i>0$. We will show that this implies heteroclinic intersection between $p_1(f_n)$ and $p_i(f_n)$, which is a contradiction.

Take $r>0$ sufficiently small, such that $B_r(p_i(f_0))\subset \cup_{x\in W^u(p_i(f_0),f_0)} \cF^s(x,f_0)$. Then
by the continuity of unstable manifolds of $p_i(\cdot)$ and the continuity of stable foliation with respect to
diffeomorphisms, there is $n_0$ such that for any $n>n_0$, any point $x\in B_r(p_i(f_n))$,
\begin{equation}\label{eq.localintersection}
\cF^s_{\loc}(x,f_n)\pitchfork W^u(p_i(f_n),f_n)\neq \emptyset.
\end{equation}

By Theorem~\ref{main.sksleton}, $p_i(f_0)\in \supp(\mu_{0,i})$ and $\mu_{0,i}(B_r(p_i(f_0)))>0$. Since $\mu_{n,1}\to \mu$ which also assigns positive measure to $B_r(p_i(f_0))$, there is $n>n_0$ such that
$\mu_{n,1}(B_r(p_i(f_0)))>0$. In particular, we have $\supp(\mu_{n,1})\cap B_r(p_i(f_0))\neq \emptyset$. Again by Theorem~\ref{main.sksleton},
$\supp(\mu_{n,1})=H(p_{1}(f_n),f_n)$, thus $\cF^s(\Orb(p_1(f_n)), f_n) \cap B_r(p_i(f_0)) \neq \emptyset$. By \eqref{eq.localintersection},
$$\cF^s(\Orb(p_1(f_n)))\pitchfork W^u(p_i(f_n),f_n)\neq \emptyset,$$
which contradicts the fact that $\{p_{1}(f_n),\cdots, p_k(f_n)\}$ is a skeleton of $f_n$ and thus by Lemma~\ref{l.skeleton}[(1)] there is no heteroclinic intersection between
$p_i(f_n)$ and $p_j(f_n)$ for $1\leq i \neq j \leq k$.
\end{proof}

To prove Theorem~\ref{main.robustskeleton}, it remains to show that for diffeomorphisms in $\Diff^{1+}(M)\cap \cU^\circ$,
the supports of corresponding physical measures and the closures of their basins vary in a lower semi-continuous fashion, both in the sense of the Hausdorff topology.

Indeed, by the unstable manifold theorem of fixed saddle, for
each $R>0$, the local invariant manifolds $W^u_R(\Orb(p_i(g),g))$
vary continuously with $g\in \cU$; moreover, the stable foliation also varies continuously with respect to $g$.
Thus the closures of $W^u(\Orb(p_i(g),g))$ and $\bigcup_{x\in W^u(\Orb(p_i(g)),g)} \cF^s(x,g)$
both vary in a lower semi-continuous fashion with $g$, relative to the Hausdorff topology. By Theorem~\ref{main.sksleton},
this means that the supports and the closures of the basins of the physical measures vary lower semi-continuously with the dynamics.
The proof of Theorem~\ref{main.robustskeleton} is now complete.

\section{Proof of Theorem~\ref{main.generic}}\label{s.8}
In this section we will generalize the result of Theorem~\ref{main.robustskeleton} to $C^1$ generic diffeomorphisms in $\cU$. The proof is similar to~\cite[Theorem B]{HYY}. The key observations are:
\begin{itemize}
	\item $C^{1+}$ diffeomorphisms are dense in $C^1$ topology;
	\item skeletons are upper semi-continuous in $\cU$;
	\item the support of physical measures for $C^{1+}$ $g\in\cU$ are homoclinic classes, which are (generically) Lyapunov stable and lower semi-continuous with the dynamics;
	\item the candidate space of physical measures, $\G(\cdot)$, is upper semi-continuously.
\end{itemize}
These properties will allow us to find a residual subset of $\cU$, consisting of continuous points of $H(p_i(\cdot),\cdot)$ and $\cG(\cdot)$. We will prove Theorem~\ref{main.generic} on this residual subset of $\cU$.

Throughout this section, let $f:M\to M$ be a $C^{1+}$ partially hyperbolic diffeomorphism with mostly expanding center,
$\cS(f)=\{p_1,\cdots, p_k\}$ be a skeleton of $f$, and $\cU$ be the $C^1$ neighborhood of
$f$ provided by Theorem~\ref{main.robustskeleton}. Recall that by Lemma~\ref{l.robustskeleton}, the  cardinality of skeleton varies in an upper semi-continuous way,
we may choose a $C^1$ open and dense subset $\cU^\circ\subset \cU$, such that the cardinality of skeleton is $C^1$ locally constant for diffeomorphisms in $\cU^\circ$.

Take any $C^{1+}$ diffeomorphism $g\in \cU^\circ$, then $g$ has $l\leq k$ physical measures due to Theorem~\ref{main.robustskeleton}. Furthermore,  there is a subset of the continuation $\cS(g) = \{p_1(g),\cdots, p_k(g)\}$
which forms a skeleton of $g$. After reordering, we may assume $\{p_1(g),\cdots,p_l(g)\}$ to be a skeleton of $g$.
Then by Lemma~\ref{l.robustskeleton}[(2)], there is a $C^1$ neighborhood $\cV\subset \cU^\circ$ of $g$,
such that for any $C^1$ diffeomorphism $h\in \cV$, $\{p_1(h),\cdots, p_l(h)\}$ forms a skeleton of $h$.

Then by Lemma~\ref{l.skeleton}[(2)], for any $C^1$ diffeomorphism $h\in \cV$ and any $1\leq i \neq j \leq l$, $W^u(\Orb(p_i(h)),h)\cap \cF^s(\Orb(p_j(h)),h)=\emptyset$.
Using Bonatti and Crovisier's connecting lemma (\cite{BC}), we see that for any diffeomorphism $h'\in \cV$ and any $1\leq i \neq j \leq l$,
$$\Cl(W^u(\Orb(p_i(h^\prime)),h'))\cap \Cl(W^s(\Orb(p_j(h^\prime)),h'))=\emptyset,$$
since otherwise one can create a non-trivial intersection between $W^u(\Orb(p_i(\cdot)),\cdot)$ and $\cF^s(\Orb(p_j(\cdot)),\cdot)$.

By Proposition~\ref{p.skeletongeometry}, $$\Cl(W^u(\Orb(p_i(h^\prime)),h'))=H(p_i(h^\prime),h^\prime)\subset \Cl(W^s(\Orb(p_j(h^\prime)),h')).$$
Thus, we have
\begin{equation}\label{eq.disjointhomoclinic}
H(p_i(h^\prime),h^\prime)\cap H(p_j(h^\prime),h^\prime)=\emptyset, \text{ and }
\end{equation}
\begin{equation}\label{eq.unstabledisjoint}
\Cl(W^u(\Orb(p_i(h^\prime),h^\prime)))\cap \Cl(W^u(\Orb(p_j(h^\prime)),h^\prime))=\emptyset.
\end{equation}

We need the following generic property proved by Morales and Pacifico [17]:
\begin{proposition}\label{p.genericlyapunov}
For every $h$ that belongs to a $C^1$ residual subset of diffeomorphisms
$\cR_0$ and every periodic point $p$ of $h$, the set $\Cl(W^u(\Orb(p),h))$ is Lyapunov stable.
\end{proposition}

Recall that the map $\cG$ which maps a diffeomorphism $h\in \cV$ to $\G(h)$ is upper semi-continuous
by Proposition~\ref{p.existencephysical}. Let $\cR_1\subset \cV$ be the residual subset of diffeomorphisms
which are continuous points of the map $\cG$. For each $1\leq i \leq l$, also
consider the map $\cI_i$ from $\cV$ to compact subsets of $M$:
$$\cI_i(h)=H(p_i(h),h).$$
Because homoclinic classes vary lower semi-continuously with respect to diffeomorphisms
(since they contain hyperbolic horseshoes), there is a residual subset of
diffeomorphisms $\cR_2 \subset \cV$ consists of the continuous points of $\cI_i$ for every $1\leq i\leq l$.
Now let us take $\cR=\cR_0\cap \cR_1\cap \cR_2\subset \cV$. We are going to show that the residual set $\cR$ satisfies the
conditions we need.

\begin{proposition}\label{p.genericphysical}
Every $C^1$ diffeomorphism $h\in \cR$ has exactly $l$ physical measures, each of which is supported on
$\Cl(W^u(\Orb(p_i(h)),h))$ for some $i=1,\cdots,k$.
Furthermore, the basin of each physical measure covers a full volume subset within
a neighborhood of its support.
\end{proposition}
\begin{proof} For any $C^{1+}$ diffeomorphism $h^\prime \in \cV$, denote by
$\mu_{h^\prime,1},\cdots, \mu_{h^\prime,l}$ the ergodic physical measures of $h^\prime$. Then by Proposition~\ref{p.existencephysical},
$\G(h^\prime)$ is the simplex generated
by $\{\mu_{h^\prime,1},\cdots, \mu_{h^\prime,l}\}$.
For any $h\in \cR$, by the continuity of the map $\cG$ at $h$, we see that
$\G(h)=\cG(h)$ is a simplex of dimension $m_h\leq l$. In particular, the number of extreme elements of
$\G(h)$ is at most $l$. Below we will show that it is in fact $l$.

Denote the extreme points of $\G(h)$ by $\mu_{h,1},\cdots, \mu_{h,m_h}$. Let
$h_n$ be a sequence of $C^{1+}$ diffeomorphisms converging to $h$ in $C^1$ topology.
By continuity of $\cG(\cdot)$ and relabelling if necessary, we may assume that $\lim \mu_{h_n,i}=\mu_{h,i}$ for $i=1,\cdots, m_h$.
Note that $\mu_{h,i}$ is supported on $W^u(\Orb(p_i(h)),h)$. This is because by Theorem~\ref{main.sksleton}, $\mu_{h_n,i}$ is supported
on $W^u(\Orb(p_i(h_n)),h_n)=H(p_i(h_n),h_n)$, and $h$ is a continuous point of the map $\Gamma_i(\cdot)$, so we must have
$\lim_{n} H(p_i(h_n),h_n)=H(p_i(h),h)$.

Next, we claim that $m_h=l$.
Assume that this is not the case. Then we take $m_h<j\leq l$ and take a weak-$*$ limit
$\mu_{h}=\lim_n \mu_{h_n,j}$. Note that $\mu_h\in \G(h)$ is supported on $W^u(\Orb(p_j(h)),h)$ by
the discussion above. Take any ergodic component $\tilde{\mu}_h$ or $\mu_h$, then $\tilde{\mu}_h\in \G(h)$
by Lemma~\ref{l.ergodicdecompo} and is still supported on $\Cl(W^u(\Orb(p_j(h))),h)$. Thus by \eqref{eq.disjointhomoclinic},
$\tilde{\mu}_h\neq \mu_{h,i}$ for every $i=1,\cdots, m_h$. We have thus created a new extreme point of $\G(h)$, which
is a contradiction.

To finish the proof, we have to show that each $\mu_{h,i}$ is a physical measure.
Since $\Cl(W^u(\Orb(p_i(h)),h))$ is Lyapunov stable, we can take $U_i\supset V_i$ open
neighborhoods for each $\Cl(W^u(\Orb(p_i(h)),h))$, such that $\{U_i\}_{i=1,\cdots,l}$ are disjoint and
for each $i$ and any $n>0$, $h^n(V_i)\subset U_i$.
By Proposition~\ref{p.physical}, there is a full volume subset $\Gamma_i\subset V_i$ such that for any $x\in \Gamma_i$,
any limit $\mu$ of the sequence $\frac{1}{n}\sum_{i=0}^{n-1}\delta_{h^i(x)}$ belongs to $\G(h)$. Note that
since $x\in V_i$, we have $h^n(x)\in U_i$ for all $n\geq 1$. As a result, $\mu$ is supported on $U_i$.
On the other hand, $\mu_{h,i}$ is the only ergodic measure in $\G(h)$ that is supported on $U_i$.
It follows that $\mu=\mu_{h,i}$. This implies that Lebesgue almost every point of $x\in V_i$ belongs to
the basin of $\mu_{h,i}$. The proof is complete.
\end{proof}

We conclude the proof of Theorem~\ref{main.generic} with the following lemma:

\begin{lemma}\label{l.fullbasin}
The basins of $\mu_{h,i}$ for $i=1,\cdots, l$ covers a full volume set.
\end{lemma}
\begin{proof}
Let $\Gamma$ be the full volume subset given by Proposition~\ref{p.physical}. We are going to show
that $\Leb(\Gamma\setminus \bigcup_{i=1}^l \cB(\mu_{h,i}))=0$.

We prove by contradiction. Write $\Lambda=\Gamma \setminus \bigcup_{i=1}^l \cB(\mu_{h,i})$ and suppose that
$\Leb(\Lambda)>0$. Let $x\in \Lambda$ be a Lebesgue density point of $\Lambda$, which means that for any $r>0$,
we have $\Leb(B_r(x)\cap \Lambda)>0$. Let $\mu$ be any limit point of the sequence $\frac{1}{n}\sum_{i=0}^{n-1}\delta_{h^i(x)}$.
Since $\mu\in \G(h)$, $\mu$ can be written as a combination of $\mu_{h,i}$:
$$\mu=a_1\mu_{h,1}+\cdots+a_l\mu_{h,l},$$
where $a_1+\cdots+a_k=1$.

Suppose without loss of generality that $a_1>0$, then $\mu(V_1)>0$ where $V_1$ is the neighborhood of $\Cl(W^u(\Orb(p_i(h)),h))$ in the proof of the previous proposition. Thus there is
$n>0$ such that $\frac{1}{n}\sum_{i=0}^{n-1}\delta_{h^i(x)}(V_1)>0$. In particular,
there is $0\leq m \leq n-1$ such that $h^m(x)\in V_1$. Take $\vep>0$ sufficiently small, we have
$h^m(B_\vep(x)) \subset V_1$. By Proposition~\ref{p.genericphysical}, $f^m(B_\vep(x)\cap \Lambda)$
intersects with the basin of $\mu_{h,1}$ on a positive
volume set. Because the basin of a measure is invariant under iteration of $h$ and
$h^{-1}$, we have $\Leb(\Lambda \cap \cB(\mu_{h,1}))>0$, which contradicts with the choice of $\Lambda$.
\end{proof}

\section{Gibbs-Markov-Young structure}\label{s.9}
To study statistical properties of some non-uniformly hyperbolic systems, in \cite{You98} Young constructed Markov towers, which are Markov partitions with infinitely many symbols and certain recurrence property. In particular she uses tower to study statistical properties of these
non-uniformly hyperbolic systems, including the existence of physical measures, exponential decay of correlations and
the validity of the Central Limit Theorem for the physical measure. These structures have some properties which address
to Gibbs states and they are nowadays commonly called as Gibbs-Markov-Young (GMY) structures.

Alves and Li in \cite{AL} obtained GMY structures for partially hyperbolic attractors and they managed
to prove the exponential decay of correlations: if the lack of expansion of the
system at time $n$ in the center-unstable direction is exponential small, then the system
has some GMY structure for physical measures with exponential decay of recurrence times.
In this section we will show that their criterion is satisfied for any physical measures
of any $C^{1+}$ diffeomorphisms with mostly expanding center.

As before, we assume $f$ to be a $C^{1+}$ partially hyperbolic diffeomorphism with mostly expanding
center, $\{p_1,\cdots,p_k\}$ be a skeleton of $f$ and $\mu_1,\cdots,\mu_k$ are the corresponding physical measures
of $f$ in the sense of Theorem~\ref{main.sksleton}. Recall that $P=\prod_{i=1}^k \pi(p_i)$.

By Corollary~\ref{maincor.period}, $\{f^{nP}\}_{n>0}$ also have mostly expanding center, and they share the same
physical measures and skeletons.  Therefore, to simply notation,  we may assume that $\{p_i\}_{i=1}^k$ are all fixed points and $P=1$. Moreover, by
Proposition~\ref{p.onestep}, we may assume that there is $b_0>0$ such that for any $\mu\in \G^u(f)$:
\begin{equation}\label{eq.fcuexpanding}
\int \log \|Df^{-1}\mid_{E^{cu}(x)}\|d\mu(x)<-b_0.
\end{equation}

The notations below are used by Alves and Li~\cite{AL} and clearly resembles our definition of hyperbolic times:
\begin{definition}
Given $b>0$, we say that $f$ is \emph{$b$ non-uniformly expanding ($b$-NUE)} at a point $x$
in the central-unstable direction if
\begin{equation}
\limsup_{n\to \infty}\frac{1}{n}\sum_{j=1}^n\log \|Df^{-1}\mid_{E^{cu}(f^{j}(x))}\|< -b.
\end{equation}
If $f$ satisfies ($b$-NUE) at some point $x$, then the \emph{expansion time function} at $x$
\begin{equation}
\cE_b(x)=\min\{ N\geq  1: \frac{1}{n}\sum^n_{i=1}\log \|Df^{-1}\mid_{E^{cu}(f^{i}(x))}\|< -b/2 \;\; \text{ for any } n\geq N\}
\end{equation}
is defined and finite. We call $\{x:\cE_b(x)>n\}$ the \emph{tail of $b/2$-hyperbolic times} (at time $n$).
\end{definition}

We need the following two propositions from \cite{AL} which play the key role in the proof of decay of correlations and center limit theorem.

\begin{proposition}\cite{AL}\label{p.decayexpanding}
Assume for $b>0$ that there is a local unstable disk $D$ of $f$ and constants $0<\tau \leq 1$, $c>0$ such that
$$\Leb_D(\cE_b>n)=O(e^{-cn^\tau}).$$
Then some power $f^l$ has an physical measure $\mu$ and there is $d>0$ such that
$$C_\mu(\phi,\psi\circ f^{ln})=O(e^{-dn^\tau})$$
for Holder continuous $\phi:M\to \mathbb{R}$ and $\psi\in L^\infty(\mu)$.
\end{proposition}

\begin{proposition}\cite{AL}\label{p.centerlimit}
Assume for $b>0$ that there is a local unstable disk $D$ of $f$ and constants $0<\tau \leq 1$, $c>0$ such that
$$\Leb_D(\cE_b>n)=O(e^{-cn^\tau}).$$
Then some power $f^l$ has an physical measure $\mu$; moreover, given any H\"older continuous function $\phi$,
the limit exists:
$$\sigma^2=\lim_{n\to \infty} \frac{1}{n}\int (\sum_{j=0}^{n-1}\phi\circ f^{jl} -n\int \phi d\mu)^2d\mu.$$
Furthermore, if $\sigma^2>0$, then there is a rate function $c(\vep)>0$ such that
$$\lim_{n\to \infty} \frac{1}{n}\log \mu(\mid \sum_{j=0}^{n-1}\phi\circ f^{jl} -n\int \phi d\mu \mid \geq \vep)=-c(\vep).$$
\end{proposition}

\begin{remark}\label{rk.decayexpanding}
From the proof, the physical measure equals to the limit of
$$\lim_{n\to \infty} \sum_{i=0}^{n-1}\Leb_{f^i(\Lambda)}$$
where $\Lambda\subset D$ is some subset with positive volume.
\end{remark}

With these notations, we are ready to prove Theorem~\ref{main.decay of Correlations} and Corollary~\ref{maincor.large deviations}. It suffices for us prove only for physical measures $\mu_1$:
Take $D=W^u_r(p_1)$. We will show in the end of this section that $D$ satisfies the following property:
\begin{lemma}\label{l.tailofhyperbolictime}
There are constants $0<\tau \leq 1$ and $c>0$ such that
$$\Leb_D(\cE_{b_0}>n)=O(e^{-cn^\tau}).$$
\end{lemma}
Then we may applying Proposition~\ref{p.decayexpanding} and Proposition \ref{p.centerlimit} on some physical measure $\mu$ for some power $f^l$ of $f$.
Moreover, by Proposition~\ref{l.support}, Lebesgue almost every point belongs to the basin of $\mu_1$, and thus by Remark~\ref{rk.decayexpanding}, for any subset $\Lambda\subset D$  with positive volume, we have
$$\lim_{n\to \infty} \sum_{i=0}^{n-1}\Leb_{f^i(\Lambda)}=\mu_1.$$
Then we conclude the proof Theorem~\ref{main.decay of Correlations} and Corollary~\ref{maincor.large deviations}.

It remains to show the proof of Lemma~\ref{l.tailofhyperbolictime}.
\begin{proof}
We need the following result:
\begin{proposition}\cite{D}[Proposition 3.1]
Let $\cB$ be any foliation box for the unstable foliation $\cF^u$ of $f$, $A$ be any H$\ddot{o}$lder function and $I_A=\{\int A d\mu\}_{\mu\in\Gibb^u(g)}$.
Then $\forall \vep>0$, $\exists \delta>0$, $C>0$ such that
for any plaque $L$ of $\cF^u\mid \cB$,
$$\Leb_L(\{x:d(\frac{1}{n}S_n(A)(x),I_A)\geq \vep\})\leq C e^{-\delta n},$$
where $S_n(A)=\sum_{i=1}^{n}A(f^i(x))$.
\end{proposition}
Fix $\cB$ to be any foliation box for the unstable foliation $\cF^u$ such that $D\subset \cB$.
By \eqref{eq.fcuexpanding}, for $A=\log \|Df^{-1}\mid_{E^{cu}(x)}\|$, $I_A\subset (\infty, -b_0)$. Applying the previous proposition with  $\vep = b_0/2$, we obtain $C>0, \delta>0$ such that for any plaque $L$ of
$\cF^u\mid \cB$,
\begin{equation}\label{eq.leaftail}
\Leb_L(\{x: \frac{1}{n} \sum_{i=1}^n \log \|Df^{-1}\mid_{E^{cu}(f^i(x))}\|\geq -b_0/2\})\leq C e^{-\delta n}.
\end{equation}

Note that
$$\{x:\cE_{b_0}>n\}\subset \bigcup_{m\geq n} \{x: \frac{1}{n} \sum_{i=1}^n \log \|Df^{-1}\mid_{E^{cu}(f^i(x))}\|\geq -b_0/2\}.$$
Thus by \eqref{eq.leaftail}, there are $C^\prime$ and $\delta^\prime$ such that for any unstable plaque $L\subset \cB$,
$$\Leb_L(\cE_{b_0}>n)\leq C^\prime e^{-\delta^\prime n}.$$

Because $D$ is the local unstable manifold at $p_1$, $\cF^u$ also induces a sub-foliation of $D$ (note that $\dim D = \dim E^{cu}$). It is well known that $\cF^u$ is absolutely continuous, so is the sub-foliation of $D$. Then
the previous inequality implies that there is $C_0>0$ such that
\begin{equation}\label{eq.disktail}
\Leb_D(\cE_{b_0}>n)\leq C_0 e^{-\delta^\prime n}.
\end{equation}

Then Lemma~\ref{l.tailofhyperbolictime} follows with $\tau =1$.
\end{proof}

\end{document}